\documentclass[12pt,reqno]{amsart}
\usepackage{amsmath, amssymb, amsthm, enumitem} 
\usepackage[pagewise]{lineno}\linenumbers
\nolinenumbers
\usepackage{mathrsfs}
\usepackage[breaklinks]{hyperref}
\usepackage{tikz}
\usepackage{mathrsfs}
\usepackage{multirow}

\usepackage{latexsym}
\usepackage{upref, eucal}
\usepackage{pgfplots}
\usepackage{rotating, tikz}
\usetikzlibrary{matrix,arrows,backgrounds}
\usepackage[headheight=110pt,top=1in, bottom=.9in, left=0.8in, right=0.8in]{geometry}

\usepackage{url}
\usepackage{amssymb}

\renewcommand{\Re}{\operatorname{Re}}

\newcommand{\e}{\epsilon}

\renewcommand{\(}{\left\(}
\renewcommand{\)}{\right\)}
\renewcommand{\[}{\left\[}
\renewcommand{\]}{\right\]}
\numberwithin{equation}{section}
\theoremstyle{plain}
\newtheorem{theorem}{Theorem}[section]
\newtheorem{lemma}[theorem]{Lemma}
\newtheorem{corollary}[theorem]{Corollary}
\newtheorem{proposition}[theorem]{Proposition}

\newtheorem{remark}[]{Remark}

\makeatletter
\def\proof{\@ifnextchar[{\@oproof}{\@nproof}}
\def\@oproof[#1][#2]{\trivlist\item[\hskip\labelsep\textit{#2 Proof of\
		#1.}~]\ignorespaces}
\def\@nproof{\trivlist\item[\hskip\labelsep\textit{Proof.}~]\ignorespaces}
\makeatother

\makeatletter
\def\@tocline#1#2#3#4#5#6#7{\relax
	\ifnum #1>\c@tocdepth % then omit
	\else
	\par \addpenalty\@secpenalty\addvspace{#2}%
	\begingroup \hyphenpenalty\@M
	\@ifempty{#4}{%
		\@tempdima\csname r@tocindent\number#1\endcsname\relax
	}{%
		\@tempdima#4\relax
	}%
	\parindent\z@ \leftskip#3\relax \advance\leftskip\@tempdima\relax
	\rightskip\@pnumwidth plus4em \parfillskip-\@pnumwidth
	#5\leavevmode\hskip-\@tempdima
	\ifcase #1
	\or\or \hskip 1em \or \hskip 2em \else \hskip 3em \fi%
	#6\nobreak\relax
	\dotfill\hbox to\@pnumwidth{\@tocpagenum{#7}}\par% <---- \dotfill -> \hfill
	\nobreak
	\endgroup
	\fi}
\makeatother

\allowdisplaybreaks
\begin{document}
	\title[Kronecker type limit formulas and Special values]{Mordell--Tornheim zeta function: Kronecker limit type formulas and Special values}
	\author{Sumukha Sathyanarayana}
	\author{N. Guru Sharan }
%	\email{neerugarsumukha@gmail.com; gurusharan.n@iitgn.ac.in}
	\thanks{
	Abbreviated title: Kronecker type limit formulas and Special values}
	\address{Department of Mathematics, Central University of Karnataka, Kadaganchi, Kalaburagi, Karnataka-585367, India.}
	\email{sumukhas@cuk.ac.in, neerugarsumukha@gmail.com}
	\address{Department of Mathematics, Indian Institute of Technology Gandhinagar, Palaj, Gandhinagar, Gujarat-382355, India}
	\curraddr{Department of Mathematics, Harish-Chandra Research Institute, a CI of Homi Bhabha National Institute, Prayagraj, Uttar Pradesh-211019, India.}
	\email{gurusharan@hri.res.in, sharanguru5@gmail.com}

	\begin{abstract}
 	In this paper, we establish Kronecker limit type formulas for the generalized Mordell--Tornheim zeta function $\Theta(r,s,t,x)$ as a function of the third variable, in terms of Riemann-zeta and Gamma values.  We also give series evaluations of $\Theta(r,s,t,x)$ in terms of Herglotz-Zagier type functions, and their derivatives. As applications of this, we derive Kronecker limit type formula in the second variable and a new infinite family of modular relations called mixed functional equations. We also study the zeroes, special values and singularities of the above function when all its arguments $r,s$ and $t$ are equal, which builds on a few earlier results due to Romik.
	\end{abstract}
	\maketitle
	\tableofcontents
	\section{Introduction}
\noindent Let us denote
\begin{align}
	\mathscr{D}:=\bigg\{(r,s,t) \in \mathbb{C}^3 \ \bigg| \Re(r+t)>1, \Re(s+t)>1 , \Re(r+s+t)>2\bigg\}. \label{Dregion}
\end{align}
 In the region $\mathscr{D},$ the Mordell-Tornheim zeta function is given by the absolutely convergent series  
	\begin{align*}
	\zeta_{\textup{MT}} (r,s,t):=\sum_{n=1}^{\infty} \sum_{m=1}^{\infty} \frac{1}{n^{r} m^{s} (n+m)^{t}}.
	\end{align*}
The function $\zeta_{\textup{MT}} (r,s,t)$ can be meromorphically continued, using the Mellin-Barnes method, to the entire $\mathbb{C}^3$ space, with its singular hyperplanes given by 
\begin{equation} \label{pole conditions}
	\left.
	\hspace{5 cm}
	\begin{aligned}
		&r+t=1-\ell, \\
		&s+t=1-\ell, \\
		&r+s+t=2,\\
	\end{aligned}
	\hspace{5 cm} \right\}
\end{equation}
for any $\ell \in \mathbb{N}\cup\{0\}$;
%by $r+t=1-\ell, \ s+t=1-\ell,$ for any $\ell \in \mathbb{N}\cup\{0\}$ and by $r+s+t=2$. 
refer to Matsumoto's paper \cite{matsumoto} for finer details of the hyperplanes. The behavior of $\zeta_{\textup{MT}} (r,s,t)$ around its singularities is largely unexplored. In a recent work with Dixit \cite{dss2}, the authors considered one of its generalization
\begin{align}
	\Theta(r,s,t,x):=\sum_{n=1}^{\infty}  \sum_{m=1}^{\infty}\frac{1}{n^{r}m^{s}(n+mx)^{t}}, \label{Tdef}
\end{align}
for $x>0$ and $(r,s,t) \in \mathscr{D}$, defined in \eqref{Dregion}. Clearly, $\Theta(r,s,t,1) =\zeta_{\textup{MT}} (r,s,t)$. One can easily observe that $\Theta(r,s,t,x)$ satisfies the following functional equations for $x>0$ and $(r,s,t) \in \mathscr{D}$:
\begin{align}
	&\Theta(r,s,t,x) =  \Theta(r-1,s,t+1,x)+x \hspace{0.1 cm} \Theta(r,s-1,t+1,x). \label{Tsplit}  \\
	&	\Theta(r,s,t,x)= x^{-t} \Theta(s,r,t,\tfrac{1}{x}).  \label{Tinv}
\end{align} 
The meromorphic continuation of $\Theta(r,s,t,x)$ can be obtained similarly as in the case of $\zeta_{\textup{MT}} (r,s,t)$ \cite[Theorem 1]{matsumoto}. Recently, it has been shown that \cite[Theorem 1.1]{dss2}, as $t \to 0$,
	\begin{align}\label{principal part equation}
		\Theta(1,1,t,x)=\frac{2}{t^2}+\frac{2\gamma-\log(x)}{t}+\gamma^2-\gamma\log(x)-\frac{\pi^2}{6}+O(|t|),
	\end{align}
	where $\gamma$ is Euler's constant. Results of this kind did not exist in the literature prior to this. Moreover, one can also derive the higher Laurent coefficients of 	\eqref{principal part equation} as a finite sum of Arakawa and Kaneko constants \cite[Corollary 5]{arakawa-kaneko} defined by
	\begin{equation*}
		c_k:=\frac{1}{k!}\int_{0}^{\infty}\left\{\log^{k}(y)-\log^{k}\left(\frac{y}{1-e^{-y}}\right)\right\}\left(\frac{1}{e^y-1}-\frac{1}{y}\right)\, dy, 
	\end{equation*}
and integrals of the form 	
\begin{align*}
	L_k^*(x):=	\sum_{j=0}^{k}a_{k+1, j}\vspace{1mm}	\int_{0}^{\infty} \frac{\log^{j}(y)}{y} \log \left( \frac{1-e^{-xy}}{1-e^{-y}} \right) \log \left( \frac{1-e^{-y}}{y}  \right) \, dt 
\end{align*}
for $ k \in \mathbb{N}$. For finer details, see the proofs of Theorem 3.4 and Proposition 3.7 in \cite{dss2}. Recently, Matsumoto, Onodera and Sahoo, have derived a asymptotic result for the Mordell-Tornheim multiple-zeta function of higher depts \cite{dilip}. 
%For zeta functions associated with quadratic forms on imaginary number field,  Kronecker obtained a limiting formula as $s \to 1$.
%Consider the real binary quadratic form $Q(x,y) =ax^2 +bx y+cy^2$ with discriminant $D= b^2-4ac$. For $a>0$, the associated zeta function is defined by
%\begin{align*}
%	\zeta_{Q}(s)=\sum_{(m,n) \ne (0,0)} Q(m,n)^{-s}
%\end{align*}
%where the sum runs over all the non-origin cordinates of the lattice $\mathbb{Z}\times\mathbb{Z}$. $\zeta_{Q}(s)$ converges absolutely for $\Re(s)>1$. For $D<0$, Kronecker proved that, around $s=1$,
%\begin{align}
%|D|^{\frac{s}{2}}	\zeta_{Q}(s)= \frac{2 \pi}{s-1} + 2 \pi \left(2 \gamma - \log 4 + H\left(\frac{-b+ \sqrt{D}}{2 a }\right)\right) + O(|s-1|), \label{KLF}
%\end{align}
%where $H(z) = - \log\left(y |\eta(z)|^4 \right),$ with $\eta(z)$ being the Dedekind eta function. 

 We call the results of the kind \eqref{principal part equation} as the Kronecker limit type formulas. Siegel \cite{siegel} connected the Kronecker limit formula with the theory of $L$-functions. Over the century, several prominent mathematicians such as Hecke \cite{hecke}, Herglotz \cite{herglotz}, Ramachandra \cite{rama}, Stark \cite{stark} and Zagier \cite{zagier} gave important contributions to this field. Note that \eqref{principal part equation} is a Kronecker limit type formula for the function $\Theta(1,1,t,x)$ in the variable $t$, around $t=0$.

In Section \ref{Sec2}, we give Kronecker limit type formulas for $\Theta(r,s,t,x)$ in $t$, i.e., the \emph{third variable}, in Theorem \ref{1stpolar}. For a fixed $r,s \in \mathbb{C}$, only possibilities of poles for $\Theta(r,s,t,x)$ in $t$ are at $t=1-r-\ell$, $t=1-s-\ell$, for $\ell \in \mathbb{N} \cup\{0\}$ and $t=2-r-s$. However, we show that all of these points need not have a pole necessarily. As a direct consequence of our Kronecker limit type formula, the singular part of the Laurent series of $\Theta(r,s,t,x)$ in $t$ is established. Interestingly, we show that the order of pole too varies by cases. $\Theta(r,s,t,x)$ can have a single or a double pole, or could also have no pole in $t$, depending on the choices of $r$ and $s$. This method cannot be used to obtain the Kronecker limit type formula for the $\Theta(r,s,t,x)$ in the \emph{second variable}. 

In the Section \ref{Section4}, we give the series evaluations of $\Theta(r,s,t, x)$ in terms of zeta values, $\psi(x)$ and its derivatives, in Theorem \ref{THM-EVAL}. As an application of our results, we use Theorem \ref{THM-EVAL} to get \emph{mixed functional equations} in Theorem \eqref{MIXEDthm} as a corollary, which have not appeared in the literature. We call them mixed functional equations since they contain terms of the form both 
\begin{align*}
	\sum_{m=1}^{\infty} \frac{\psi^{(a)}(mx+1)}{m^{b}}\hspace{1cm}\text{ and} \hspace{1cm}
	\sum_{m=1}^{\infty} \frac{\gamma + \psi(mx+1)}{m^{c}},
\end{align*}
for $a,b,$ and $c \in \mathbb{N}$, which appear in the Guinand's \cite{apg3} and the Vlasenko-Zagier's \cite[Proposition 4]{vz} functional equations, respectively. As explained in Remark \ref{REMARK1}, of the infinite family of new identites we obtain, we now illustrate one below, 
\begin{align}
	-\sum_{m=1}^{\infty} \frac{\psi'(mx+1)}{m^2} + 	\sum_{m=1}^{\infty} \frac{\psi'(\frac{m}{x}+1)}{m^2} + \frac{2}{x} \sum_{m=1}^{\infty} \frac{\gamma+\psi(mx+1)}{m^3} = \zeta^2(2). \label{newIDex}
\end{align} 

Also as a corollary of Theorem \ref{THM-EVAL}, we derive a Kronecker limit type formula, as stated in Theorem \eqref{KLF in s}, for $\Theta(r,s,t, x)$ in the \emph{second variable}.

As introduced by Zagier \cite[Section 7]{zagier-witten}, a specific case of the Mordell-Tornheim zeta function $\zeta_{\textup{MT}}(r,s,t)$ appears as the \textit{Witten-zeta function} of the Lie Algebra $A_2=\textup{sl}(3,\mathbb{C})$. This has recently been studied in detail by Romik \cite{romik}. He studied the Witten-zeta function associated to the group $SU(3)$, denoted by $\omega(s)$, given by,
\begin{align}
	\omega(s) :=  \sum_{j,k\geq 1} (2 \textup{ dim } W_{j,k})^{-s} = \sum_{j,k\geq 1} \frac{1}{(jk(j+k))^s} = \zeta_{\textup{MT}}(s,s,s) = \Theta(s,s,s,1). \label{w defn}
\end{align}
He further showed the following \cite[Theorems 1.2 and 1.3]{romik}:
\begin{enumerate}[label=\roman*.]
	\item The series in \eqref{w defn} is absolutely convergent for Re$(s)>2/3$, and defines a holomorphic function in the region.
	\item $\omega(s)$ can be meromorphically extended to a function on $\mathbb{C}$.
	\item  $\omega(s)$ has a simple pole at $s=2/3$ with residue $\frac{1}{2\pi \sqrt{3}} \ \Gamma^3 \hspace{-0.9 mm}\left( \frac{1}{3} \right)$, a simple pole at $s=1/2-k$ (for any non-negative integer $k$) with residue $(-1)^k 2^{-4k} \binom{2k}{k} \zeta\left( \frac{1-6k}{2} \right)$. Also, $\omega(s)$ has no other points of singularity. 
	\item It has the following special values at  non-positive integer arguments:
	\begin{align}
		\omega(0) &= \frac{1}{3}, \label{sv omega1} \\
		\omega(-n) &= 0, \qquad (\textup{for any positive integer } n). \label{sv omega2}
	\end{align}
\end{enumerate}

Onodera showed the existence of a refinement in the zeros of $\omega(s)$. As given in \cite[Proposition 4.4]{onodera}, it is known that $\omega(s)$ has a simple zero at odd negative integers but has a zero of order $2$ at the even negative ones. Very recently, a much general structure for the negative even zeros has been proven \cite{KCA}. It was shown that the Witten-zeta function $\zeta_{\Phi}(s)$, for a root system $\Phi$ associated to a simple Lie algebra $\mathfrak{g}$, has a zero at $s=m$ (any negative even integer $m$) of order as follows \cite[Theorem 3.6]{KCA}:
\begin{align*}
	\textit{ord}_{s=m} (\zeta_{\Phi}(s)) \geq \textup{ Rank of }\Phi.
\end{align*}  

It remains curious to understand reasons for the difference in the order of zeros, depending on the parity. With this literature in the background, we study the zeros (and other special values) of the function $\Theta(s,s,s,x)$, which is by itself a generalization of $\omega(s)$, in Section \ref{RomikSection}. We prove a new parity based stark difference in Theorem \ref{SV}, that $\Theta(s,s,s,x)$ still vanishes at negative odd integers $s$, but is non--zero for negative even integers $s$. 

One might have observed that the argument (say $n=2$) in \eqref{sv omega2} lies of on the hyperplanes listed in \eqref{pole conditions}. In the case of meromorphic functions of multiple complex variables, the singularities could either be poles (where the magnitude blows up) or \textit{points of indeterminacy} (where different limiting values could be obtained by approaching the singularity in different directions); refer to \cite[Chapter 15]{Gauthier} and \cite[Section 4.2.2]{boas} for detailed explanation.

Special values at few points of indeterminacy of $\zeta_{\textup{MT}}(r,s,t)$ have been studied before; for instance, one could read \cite[Section 6]{Komori} and \cite[Section 3]{AET} and their citations for existing literature on this topic. All the special values listed below for $\Theta(r,s,t,x)$ are at the points of indeterminacy. \textit{One thus needs to fix a way of approaching the singularity}, hence we stick to the following limiting approach throughout Section \ref{RomikSection}; for a point of indeterminacy $(s_0, s_0, s_0,x)$, we define
\begin{align}
	\Theta(s_0, s_0, s_0, x)_{L_1} &:= \lim_{s \to s_0} \Theta(s,s,s,x). \label{sense}
	%	&=\lim_{t \to t_0} \lim_{r \to r_0} \lim_{s \to s_0} \Theta(r, s ,t, x)\\
	%	&=\lim_{t \to t_0} \lim_{s \to s_0} \lim_{r \to r_0} \Theta(r, s ,t, x).
\end{align}
The above equality can also be understood as the limit of $\Theta(s,s,s,x)$, as a function of one complex variable $s$. For instance, in the sense of \eqref{sense}, we list the first four non-zero special values below:
%\begin{align*}
%	\Theta(-2,-2,-2,x) &= -\frac{(1-x^4)^2}{14400x^3},\\ 
%	\Theta(-4,-4,-4,x) &= \frac{1-x^6-x^8+x^{14}}{15120 x^5},\\
%	\Theta(-6,-6,-6,x) &= \frac{-174611 + 83611 x^8 + 182000 x^{10} + 83611 x^{12}- 174611 x^{20}}{158558400 x^{7}}.	
%\end{align*}
\begin{equation} \label{svT3eg}
	\left.
	\hspace{0.5 cm}
	\begin{aligned}
		\Theta(0,0,0,x)_{L_1} &= \frac{1+6x+x^2}{24x},\\
		\Theta(-2,-2,-2,x)_{L_1} &= \frac{-1+2x^4-x^8}{14400x^3}, \\
		\Theta(-4,-4,-4,x)_{L_1} &= \frac{1-x^6-x^8+x^{14}}{15120 x^5}, \\
		\Theta(-6,-6,-6,x)_{L_1} &= \frac{-174611 + 83611 x^8 + 182000 x^{10} + 83611 x^{12}- 174611 x^{20}}{158558400 x^{7}}.	\\
	\end{aligned}
	\hspace{0.6 cm}
	 \right\}
\end{equation}
We can see that $\Theta(0,0,0,1)_{L_1}=1/3$, which has been shown before by Romik, as explained in \eqref{sv omega1}. Observe that the right-hand side of the above expressions \eqref{svT3eg}, except the first one, vanish upon letting $x=1$. This phenomenon is explained in Remark \ref{RMK-jzero}, which indirectly establishes the above mentioned result of Romik \eqref{sv omega2}. In Corollary \ref{corollary0}, we further more prove a larger class of zeroes of $\Theta(r,s,t,x)$.

	We can also specialize values in Theorems \ref{1stpolar} and \ref{3rdpolar} to get numerous special values with the limiting approach 
\begin{align}
\Theta(r_0,s_0,t_0,x)_{L_2} := \lim_{t \to t_0} \Theta(r_0,s_0,t,x). \label{sense2}
\end{align} 
For instance, in the sense given in \eqref{sense2}, we can obtain other non-zero special values using \eqref{T1C3}, as listed below:
%\begin{align*}
%\Theta(-2,-2,-2,x)=\lim_{t \to -2} \Theta(-2,-2,t,x)=\frac{-1+x^4-x^8}{7200x^3}.
%\end{align*}
\begin{equation} \label{svT4eg}
	\left.
	\hspace{0.5 cm}
	\begin{aligned}
		\Theta(0,0,0,x)_{L_2} &= \frac{2+6x+2x^2}{24x}, \\
		\Theta(-2,-2,-2,x)_{L_2} &= \frac{-2+2x^4-2x^8}{14400x^3}, \\
		\Theta(-4,-4,-4,x)_{L_2} &= \frac{2-x^6-x^8+2x^{14}}{15120 x^5}, \\
		\Theta(-6,-6,-6,x)_{L_2} &= \frac{-349222 + 83611 x^8 + 182000 x^{10} + 83611 x^{12}- 349222x^{20}}{158558400 x^{7}}.	\\
	\end{aligned}
	\hspace{0.6 cm}
	\right\}
\end{equation}
We can see that $\Theta(0,0,0,1)_{L_2}=5/12$, which has been shown before by Komori \cite[Equation (6.16)]{Komori}. One must also observe the similarity between the special values given in \eqref{svT3eg} and \eqref{svT4eg} at various points of indeterminacy, obtained due to different limiting approaches. 

For the depth $3$ multi-zeta function $\zeta_3(r_1,r_2,r_3)$, Akiyama, Egami and Tanigawa proved the special value in \cite[Remark 3]{AET} that
\begin{align*}
	\zeta_3(4,-3,-2) = -\frac{461}{2520}-\frac{\pi^2}{144}+\frac{\pi^4}{45360}+\frac{\zeta(3)}{420}.
\end{align*}
Using our methods, we can also find such \textit{mixed non-positive} special values for $\Theta(r,s,t,x)$.
%\begin{itemize}
%	\item Romik $\zeta_{\textup{MT}}(s,s,s)$
%	\item Lie algebra and representation theory
%	\item Non-positive special values (and points of indeterminacy)
%	\item Trivial zeros analogous to that of $\zeta(s)$ (and non-trivial zeros?)
%\end{itemize}
	\section{Kronecker limit type formula for $\Theta(r,s,t,x)$ in the third variable} \label{Sec2}
We first prove the Proposition \ref{first}, and its corollary \ref{second}, before using them to prove the main results of this section, stated below in three Theorems \ref{1stpolar}, \ref{2ndpolar} and \ref{3rdpolar}. Theorems \ref{1stpolar} and \ref{2ndpolar} are subdivided into four cases each, whose casewise classification can be better understood by the flowcharts in Figure \ref{ExPic6} and \ref{ExPic7} respectively. Theorems \ref{1stpolar} and \ref{2ndpolar} pertains to singularities of the type $r+t=1-\ell$ for $\ell \in \mathbb{N} \cup \{0\}$ while Theorem \ref{3rdpolar} pertains to singularities of the type $r+s+t=2$.

	\begin{figure}[h!]
	\centering
	\begin{minipage}{.5\textwidth}
		\centering
		\begin{tikzpicture}[description/.style={fill=white,inner sep=2pt}]
			\useasboundingbox (2.3,3.5) rectangle (9.8,9.5);
			\scope[transform canvas={scale=0.8}]
			\node at (6,10.5) {$s \in \mathbb{C}$};
			\draw (5.4,10.2) rectangle (6.6,10.8);
			\node at (4,9) {$s \notin \mathbb{Z}$};
			\draw (3.4,8.7) rectangle (4.6,9.3);
			\node at (8,9) {$s \in \mathbb{Z}$};
			\draw (7.4,8.7) rectangle (8.6,9.3);
			\node at (6,7.5) {$s> r+\ell$};
			\draw (5.15,7.2) rectangle (6.85,7.8);
			\node at (10,7.5) {$s\leq r+\ell$};
			\draw (9.15,7.2) rectangle (10.85,7.8);
			\node at (5,6) {$s\ne \ell+1$};
			\draw (4.2,5.7) rectangle (5.8,6.3);
			\node at (7,6) {$s=\ell+1$};
			\draw (6.2,5.7) rectangle (7.8,6.3);
			\node at (9,6) {$s \in \mathbb{Z}\hspace{-0.07 cm} \setminus \hspace{-0.07 cm}\mathbb{N}$};
			\draw (8.2,5.7) rectangle (9.8,6.3);
			\node at (11,6) {$s\in \mathbb{N}$};
			\draw (10.2,5.7) rectangle (11.8,6.3);
			\draw (6,9.75) -- (6,10.2);
			\draw (6,9.75) -| (8,9.3);
			\draw (6,9.75) -| (4,9.3);
			\draw (8,8.7) -- (8,8.25);
			\draw (8,8.25) -| (10,7.8);
			\draw (8,8.25) -| (6,7.8);
			\draw (10,6.75) -- (10,7.2);
			\draw (10,6.75) -| (11,6.3);
			\draw (10,6.75) -| (9,6.3);
			\draw (6,6.75) -- (6,7.2);
			\draw (6,6.75) -| (7,6.3);
			\draw (6,6.75) -| (5,6.3);
			\node at (4,8.4) {Case I};
			\node at (7,5.4) {Case II};
			\node at (5,5.4) {Case I};
			\node at (9,5.4) {Case III};
			\node at (11,5.4) {Case IV};
			\endscope
		\end{tikzpicture}
		\caption{The cases for $r \in \mathbb{Z} \setminus \mathbb{N}$.}
		\label{ExPic6}
	\end{minipage}%
	\begin{minipage}{.5\textwidth}
		%	\end{figure}
	%	\begin{figure}[h!]
		\centering
		\begin{tikzpicture}[description/.style={fill=white,inner sep=2pt}]
			\useasboundingbox (2.3,3) rectangle (10.8,9);
			\scope[transform canvas={scale=0.8}]
			\node at (6,10.5) {$s \in \mathbb{C}$};
			\draw (5.4,10.2) rectangle (6.6,10.8);
			\node at (4,9) {$s \notin \mathbb{Z}$};
			\draw (3.4,8.7) rectangle (4.6,9.3);
			\node at (8,9) {$s \in \mathbb{Z}$};
			\draw (7.4,8.7) rectangle (8.6,9.3);
			\node at (6,7.5) {$s> r+\ell$};
			\draw (5.15,7.2) rectangle (6.85,7.8);
			\node at (10,7.5) {$s\leq r+\ell$};
			\draw (9.15,7.2) rectangle (10.85,7.8);
			\node at (8.5,6) {$s \in \mathbb{Z}\hspace{-0.07 cm} \setminus \hspace{-0.07 cm}\mathbb{N}$};
			\draw (7.7,5.7) rectangle (9.3,6.3);
			\node at (11.5,6) {$s\in \mathbb{N}$};
			\draw (10.7,5.7) rectangle (12.3,6.3);
			\node at (10.5,4.5) {$s\ne \ell+1$};
			\draw (9.7,4.2) rectangle (11.3,4.8);
			\node at (12.5,4.5) {$s=\ell+1$};
			\draw (11.7,4.2) rectangle (13.3,4.8);
			\draw (6,9.75) -- (6,10.2);
			\draw (6,9.75) -| (8,9.3);
			\draw (6,9.75) -| (4,9.3);
			\draw (8,8.7) -- (8,8.25);
			\draw (8,8.25) -| (10,7.8);
			\draw (8,8.25) -| (6,7.8);
			\draw (10,6.75) -- (10,7.2);
			\draw (10,6.75) -| (11.5,6.3);
			\draw (10,6.75) -| (8.5,6.3);
			\draw (11.5,5.7) -- (11.5,5.25);
			\draw (11.5,5.25) -| (12.5,4.8);
			\draw (11.5,5.25) -| (10.5,4.8);
			\node at (4,8.4) {Case I};
			\node at (6,6.9) {Case I};
			\node at (8.5,5.4) {Case II};
			\node at (10.5,3.9) {Case III};
			\node at (12.5,3.9) {Case IV};
			\endscope
		\end{tikzpicture}
		\caption{The cases for $r \in \mathbb{N}$.}
		\label{ExPic7}
	\end{minipage}
\end{figure}

	\begin{proposition}\label{first}
		Let $r,s,t \in \mathbb{C}$ such that $r \not\in \mathbb{N}$, and \\ $\Re(t)> \textup{max} \, \{0,1- \Re(r), 1- \Re(s), 2-\Re(r)-\Re(s)\}$. For any $M \in \mathbb{N} \cup \{0\} $ and $x>0$, we have
		\begin{align}\label{PE}
			&\frac{1}{\Gamma(t)} \int_{0}^{\infty} y^{t-1} \textup{Li}_s (e^{-xy}) \Bigg( \textup{ Li}_r (e^{-y}) - \Gamma(1-r) y^{r-1} - \sum_{\substack{k=0 }}^{M} (-1)^k\zeta(r-k) \frac{y^k}{k!}\Bigg) \, dy \notag \\
			&=  \Theta(r,s,t,x) - \frac{\Gamma(1-r)\Gamma(t+r-1)}{ x^{t+r-1}\Gamma(t)}\zeta(t+s+r-1)    - \sum_{\substack{k=0 }}^{M} \frac{ (-1)^k}{k!} \frac{\zeta(r-k)\Gamma(t+k)}{ x^{t+k}\Gamma(t)} \zeta(t+k+s).
		\end{align} 
	\end{proposition}
	
	\begin{proof}
		For $\Re(t),x,n>0$, the definition of the Gamma function $\Gamma(t)$ upon a change of variable gives,
		\begin{align*}
			\int_{0}^{\infty} y^{t-1} e^{-(x+n)y} \, dy = (x+n)^{-t} \Gamma(t).
		\end{align*} 
		Since $\textup{Re}(r+t)>1$, divide both sides of the equation by $n^r$ and sum over $n \in \mathbb{N}$ and interchange the order of summation and integration to get,
		\begin{align}
			%			\sum_{n=1}^{\infty} \int_{0}^{\infty} y^{t-1} e^{-xy} \cdot \frac{e^{-ny}}{n^r} \, dy &= \Gamma(t) \sum_{n=1}^{\infty} \frac{1}{n^r (x+n)^{t}} \notag \\
			%			\int_{0}^{\infty} y^{t-1} e^{-xy} \sum_{n=1}^{\infty} \frac{e^{-ny}}{n^r} \, dy &= \Gamma(t) \sum_{n=1}^{\infty} \frac{1}{n^r(x+n)^{t}} \notag \\
			\int_{0}^{\infty} y^{t-1} e^{-xy} \textup{ Li}_r (e^{-y})  \, dy &= \Gamma(t) \sum_{n=1}^{\infty} \frac{1}{n^r (x+n)^{t}}, \label{SI}
		\end{align} 
		where $\textup{Li}_r(z):=\sum_{u=1}^{\infty} \frac{z^u}{u^r}$ is the polylogarithm function.  For $r \not\in \mathbb{N}$, around $y=0$, as stated in \cite[Equation 8, p.29]{Erdelyi}, we have,
		\begin{align}\label{LiE}
			\textup{ Li}_r (e^{-y}) = \Gamma(1-r) y^{r-1}+ \sum_{\substack{k=0 }}^\infty (-1)^k\zeta(r-k) \frac{y^k}{k!} = O(y^{\textup{max}(0,\textup{Re}(r)-1)}). 
		\end{align}
		The bound in \eqref{LiE} along with $\Re(t)> \textup{max} \, \{0, 1-\Re(r)\}$ ensures the absolute convergence, which in turn allows the interchange using  \cite[Theorem 2.1]{temme}. 
		%		Use \eqref{SI} to get
		%		\begin{align}\label{add notnat}
			%			&\int_{0}^{\infty} y^{t-1} e^{-xy} \Bigg( \textup{ Li}_r (e^{-y})- \Gamma(1-r) y^{r-1} - \sum_{\substack{k=0 }}^{M} (-1)^k\zeta(r-k) \frac{y^k}{k!}\Bigg) \, dy \notag \\
			%			&= \Gamma(t) \sum_{n=1}^{\infty} \frac{1}{n^r (x+n)^{t}}  -\int_{0}^{\infty} y^{t-1} e^{-xy} \Bigg( \Gamma(1-r) y^{r-1} + \sum_{\substack{k=0 }}^{M} (-1)^k\zeta(r-k) \frac{y^k}{k!} \Bigg) \, dy
			%		\end{align}
		%	 \textcolor{red}{Is this step necessary?}
		%		\begin{align}\label{eva notnat}
			%			&\int_{0}^{\infty} y^{t-1} e^{-xy} \Bigg( \Gamma(1-r) y^{r-1} + \sum_{\substack{k=0 }}^{M} (-1)^k\zeta(r-k) \frac{y^k}{k!} \Bigg) \, dy \notag\\
			%			&=   \frac{\Gamma(1-r)\Gamma(t+r-1)}{x^{t+r-1}}  + \sum_{\substack{k=0 }}^{M} \frac{ (-1)^k}{k!} \frac{\zeta(r-k)\Gamma(t+k)}{x^{t+k}} . 
			%		\end{align}
		Using \eqref{SI} and the definition of $\Gamma(t)$ twice, we see that,
		\begin{align}\label{after eva notnat}
			&\int_{0}^{\infty} y^{t-1} e^{-xy} \Bigg( \textup{ Li}_r (e^{-y})- \Gamma(1-r) y^{r-1} - \sum_{\substack{k=0 }}^{M} (-1)^k\zeta(r-k) \frac{y^k}{k!}\Bigg) \, dy \notag \\
			&= \Gamma(t) \sum_{n=1}^{\infty} \frac{1}{n^r (x+n)^{t}} - \frac{\Gamma(1-r)\Gamma(t+r-1)}{x^{t+r-1}}  - \sum_{\substack{k=0 }}^{M} \frac{ (-1)^k}{k!} \frac{\zeta(r-k)\Gamma(t+k)}{x^{t+k}}.
		\end{align}
		Replace $x$ by $mx$ in \eqref{after eva notnat}, divide both the sides of resulting equation by $m^s \Gamma(t)$, sum over $m \in \mathbb{N}$ and then interchange the order of summation and integration to get,
		\begin{align}\label{final not nat}
			&\frac{1}{\Gamma(t)} \int_{0}^{\infty} y^{t-1} \textup{Li}_s (e^{-xy}) \Bigg( \textup{ Li}_r (e^{-y}) - \Gamma(1-r) y^{r-1} - \sum_{\substack{k=0 }}^{M} (-1)^k\zeta(r-k) \frac{y^k}{k!}\Bigg) \, dy \notag \\
			%			&=  \sum_{n=1}^{\infty} \sum_{m=1}^{\infty} \frac{1}{n^r m^s (n+mx)^{t}}  - \frac{\Gamma(1-r)\Gamma(t+r-1)}{ x^{t+r-1}\Gamma(t)}\zeta(t+s+r-1)    - \sum_{\substack{k=0 }}^{M} \frac{ (-1)^k}{k!} \frac{\zeta(r-k)\Gamma(t+k)}{ x^{t+k}\Gamma(t)} \zeta(t+k+s)\notag \\
			&= \Theta(r,s,t, x)  - \frac{\Gamma(1-r)\Gamma(t+r-1)}{ x^{t+r-1}\Gamma(t)}\zeta(t+s+r-1)    - \sum_{\substack{k=0 }}^{M} \frac{ (-1)^k}{k!} \frac{\zeta(r-k)\Gamma(t+k)}{ x^{t+k}\Gamma(t)} \zeta(t+k+s).
		\end{align} 
		We justify the interchange of the order of summation and integration as follows: From \eqref{LiE}, we have,
		\begin{align}\label{BOr}
			\textup{ Li}_r (e^{-y}) - \Gamma(1-r) y^{r-1} - \sum_{\substack{k=0 }}^{M} (-1)^k\zeta(r-k) \frac{y^k}{k!} =  O(y^{M+1}), \hspace{0.5 cm} \textup{as}\quad y \to 0^+.
		\end{align}
		Similarly,
		\begin{equation}\label{BOs}
			\textup{ Li}_s (e^{-y})=
			\begin{cases}
				O(1),\hspace{5mm}&\text{if}\hspace{1mm}\Re(s) \ge 1,\\
				O(y^{\Re(s)-1}),&\text{if}\hspace{1mm}\Re(s) < 1,
			\end{cases}
			\hspace{0.5 cm}\textup{as} \quad y \to 0^+.
		\end{equation}
		%	\begin{align}
			%		y^{t-1} \textup{Li}_s (e^{-xy}) =O(y^{\epsilon-1}), \textup{ for some } \epsilon >0, \, \textup{as}\quad y \to 0^+. 
			%	\end{align}
		Since $\Re(t)>\text{max}\, \{ 0, 1-\Re(s)\}$, from \eqref{BOr} and \eqref{BOs} we get,
		\begin{align}\label{BO0}
			y^{t-1} \textup{Li}_s (e^{-xy}) \bigg( \hspace{-0.2 cm} \textup{ Li}_r (e^{-y}) - \Gamma(1-r) y^{r-1} - \sum_{\substack{k=0 }}^{M} (-1)^k\zeta(r-k) \frac{y^k}{k!}\bigg) = O\Big(y^{M+\epsilon }\Big), 
		\end{align}
		for any $\epsilon >0$, as $y \to 0^+$. Also,  $\textup{ Li}_s (e^{-y})= O(y^{-K})$, for any $K>1$ as $ y \to \infty$. Therefore,  
		\begin{align}\label{BO1}
			y^{t-1} \textup{Li}_s (e^{-xy}) \bigg( \hspace{-0.15 cm} \textup{ Li}_r (e^{-y}) - \Gamma(1-r) y^{r-1} - \sum_{\substack{k=0 }}^{M} (-1)^k\zeta(r-k) \frac{y^k}{k!}\bigg) = O\Big(y^{\textup{max}\{\Re(t)+M-1-K,\, \Re(t)+r-2-K\} }\Big),
		\end{align}
		for any $K>1$, as $ y \to \infty$. Using \cite[Theorem 2.1]{temme} with the facts in \eqref{BO0} and \eqref{BO1}, we justify the interchange in \eqref{final not nat} to complete the proof.
	\end{proof}
%	\raggedright
	We now give the following result, analogous to Proposition \ref{first} for $r \in\mathbb{N}$ as a Corollary. 
	\begin{corollary} \label{second}
		Let $r \in \mathbb{N}$ and $s,t \in \mathbb{C}$ such that $\Re(t)> \textup{max} \, \{0,1- \Re(r), 1- \Re(s), 2-\Re(r)-\Re(s)\}$. For any $M \in \mathbb{N} \ \cup \{0\} $ with $M \ge r-1$, and $x>0$ we have,
		\begin{align}
			&\frac{1}{\Gamma(t)} \int_{0}^{\infty} y^{t-1} \textup{Li}_s (e^{-xy}) \Bigg( \textup{ Li}_r (e^{-y}) - (-1)^{r-1} (H_{r-1} -\log(y)) \frac{y^{r-1}}{(r-1)!} - \sum_{\substack{k=0 \\ k\ne r-1}}^{M} (-1)^k\zeta(r-k) \frac{y^k}{k!} \Bigg) \, dy \notag \\
			&=  \Theta(r,s,t,x)  - \sum_{\substack{k=0 \\ k\ne r-1}}^{M} \frac{ (-1)^k}{k!} \frac{\zeta(r-k)\Gamma(t+k)}{ x^{t+k}\Gamma(t)} \zeta(t+k+s) \notag  \\
			& \quad +  \frac{(-1)^{r}\Gamma(t+r-1)}{x^{t+r-1} (r-1)!\Gamma(t)}\Big(\zeta(t+s+r-1) \left( H_{r-1}  - \psi(t+r-1) + \log(x) \right)- \zeta'(t+s+r-1)\Big). \label{CE}
		\end{align} 
	\end{corollary}
	
	\begin{proof}
		Let $a\in \mathbb{N}$. Using the Laurent series expansions of $\Gamma(1-r)$ and $\zeta(r-a+1)$ around $r=a$ and simplifying, we get,
		\begin{align}
			&\lim_{r \to a} \Bigg( \textup{ Li}_r (e^{-y}) - \Gamma(1-r) y^{r-1} - \sum_{\substack{k=0 }}^{M} (-1)^k\zeta(r-k) \frac{y^k}{k!}\Bigg) \notag \\ 
			%			&=\int_{0}^{\infty} y^{t-1} \textup{Li}_s (e^{-xy}) \lim_{r \to a}\Bigg( \textup{ Li}_r (e^{-y}) - \Gamma(1-r) y^{r-1} - \sum_{\substack{k=0 }}^{M} (-1)^k\zeta(r-k) \frac{y^k}{k!}\Bigg) \, dy \notag\\
			%			&=\int_{0}^{\infty} y^{t-1} \textup{Li}_s (e^{-xy})\Bigg( \textup{ Li}_a (e^{-y}) -  \lim_{r \to a} \Bigg(\Gamma(1-r) y^{r-1} + (-1)^{a-1} \zeta(r-a+1) \frac{y^{a-1}}{(a-1)!} \Bigg) \notag \\
			%			&\quad \quad - \sum_{\substack{k=0\\ k \ne a-1 }}^{M} (-1)^k\zeta(r-k) \frac{y^k}{k!}\Bigg) \, dy \notag\\
			&= \textup{ Li}_a (e^{-y}) + (-1)^{a} (H_{a-1} -\log(y)) \frac{y^{a-1}}{(a-1)!} - \sum_{\substack{k=0 \\ k\ne a-1}}^{M} (-1)^k\zeta(a-k) \frac{y^k}{k!} \label{rLim}
		\end{align}
		%		 Take limit as $r \to a$ on both sides of \eqref{PE}.
		%		 Invoke \cite[Theorem 2.2]{temme} with the bounds given in \eqref{BO0} and \eqref{BO1} to see the analyticity in $r$  of the integral in the strip $a-1< \Re(r)<a+1$ with removable singularity at $a$.
		Take limit as $r \to a$ on both sides of \eqref{PE}. Invoke \cite[Theorem 2.2]{temme} with the bounds given in \eqref{BO0} and \eqref{BO1} to see the integral on the left-hand side is analytic in $r$, in the strip $a-1< \Re(r)<a+1$ with a removable singularity at $a$. Hence, interchanging the order of limit and integration, we get,
		%		 Therefore, we have
		%		where we have used the well-known \textcolor{red}{cite} Laurent series expansion of $\zeta(r-a+1)$ and $\Gamma(1-r)$ around $r=a$.
		%		Therefore letting $t \to a $ in \eqref{PE},we get 
		\begin{align*}
			&\frac{1}{\Gamma(t)} \int_{0}^{\infty} y^{t-1} \textup{Li}_s (e^{-xy}) \bigg( \textup{Li}_a (e^{-y}) + (-1)^{a} (H_{a-1} -\log(y)) \frac{y^{a-1}}{(a-1)!} - \sum_{\substack{k=0 \\ k\ne a-1}}^{M} (-1)^k\zeta(a-k) \frac{y^k}{k!} \bigg) \, dy \notag \\
			&= \Theta(a,s,t,x) - \sum_{\substack{k=0  \\ k \ne a-1}}^{M} \frac{ (-1)^k}{k!} \frac{\zeta(a-k)\Gamma(t+k)}{ x^{t+k}\Gamma(t)} \zeta(s+t+k)   \notag\\
			& - \lim_{r \to a} \bigg( \frac{\Gamma(1-r)\Gamma(t+r-1)}{ x^{t+r-1}\Gamma(t)}\zeta(s+t+r-1)- \frac{(-1)^{a}\zeta(r-a+1)\Gamma(t+a-1)}{ x^{t+a-1}(a-1)!\Gamma(t)} \zeta(s+t+a-1) \bigg)  .
		\end{align*}
		To evaluate the limit in the above equation, use the Laurent series expansions of the respective functions around $r=a$. Finally, replace $a$ by $r$ in the resultant to get \eqref{CE}.
	\end{proof}
	 
	We are now equipped to prove the Kronecker Limit type formula in \emph{third variable}, around $t=1-r-\ell$. 
	
	\begin{theorem}\label{1stpolar} \rm
		\textit{Let $r \in \mathbb{Z}\setminus \mathbb{N}$ and $\ell \in \mathbb{N} \cup \{0\}$ be arbitrary such that $\ell\geq 1-r$. Then, the following hold:}
%		\vspace{0.2 cm}
\begin{enumerate}
	\item 
		 \textit{For a fixed $s \in \mathbb{Z}$ such that $s>r+\ell$ and $s\ne \ell+1$, or $s \in \mathbb{C} \setminus \mathbb{Z}$, we have, around $t=1-r-\ell$,}
		\begin{align}\label{T1C1}
			\Theta(r,s,t,x) &= (-1)^{r-1}\frac{ (r+\ell-1)!}{ \ell!}x^\ell (-r)! \zeta(s-\ell) \notag \\
			& \quad + \sum_{\substack{k=0}}^{r+\ell-1} \frac{(r+\ell-1)!\zeta(r-k)}{ x^{1-r-\ell+k}k!(r+\ell-k-1)!} \zeta(-r+s-\ell+k+1) + O(|t-(1-r-\ell)|) 
		\end{align}
	\item	\textit{ For $s=\ell+1$, we have,  around $t=1-r-\ell$,}
		\begin{align}
			\Theta(r,s,t,x)
			&= (-1)^{r-1}(-r)! x^\ell  \frac{ (r+\ell-1)!}{\ell! (t-(1-r-\ell))} +(-1)^{r-1} (-r)! x^\ell \frac{ (r+\ell-1)!}{\ell! }\left(\gamma-\log(x)+H_{\ell}-H_{r+\ell-1}\right) \notag \\
			& \quad + \sum_{\substack{k=0}}^{r+\ell-1} \binom{r+\ell-1}{k} \frac{\zeta(r-k)}{ x^{k-r-\ell+1}} \zeta(k-r+2) + O(|t-(1-r-\ell)|). \label{T1C2}
		\end{align}
	\item 	\textit{ For a fixed $s \in \mathbb{Z} \setminus \mathbb{N}$ such that $s \leq r+\ell$, we have,  around $t=1-r-\ell$,}
		\begin{align}
			\Theta(r,s,t,x)
			&= (-1)^{r-1}x^{\ell}\frac{(-r)!(r+\ell-1)!}{ \ell!}\zeta(s-\ell)   + \sum_{\substack{k=0}}^{r+\ell-1} \frac{\zeta(r-k)(r+\ell-1)!}{ x^{k-r-\ell+1}k!(r+\ell-k-1)!} \zeta(s+1-r-\ell+k) \notag\\
			&\quad +(-1)^{s-1}x^{s-1}\frac{(r+\ell-1)!(-s)!}{ (r+\ell-s)!} \zeta(s-\ell)+ O(|t-(1-r-\ell)|). \label{T1C3}
		\end{align}
 \item 	 \textit{For a fixed $s \in \mathbb{N}$ such that $s \leq r+\ell$, we have,  around $t=1-r-\ell$,}
		\begin{align}
			&\Theta(r,s,t,x) =   x^{s-1} \frac{(r+\ell-1)!}{ (r+\ell-s)! (s-1)!}   \frac{\zeta(s-\ell)}{(t-(1-r-\ell))} \notag \\
			&\quad + x^{s-1} \frac{(r+\ell-1)!}{ (r+\ell-s)! (s-1)!} (\gamma -H_{r+\ell-1}+H_{s-1}-\log(x)) \zeta(s-\ell) -(-1)^{r}x^{\ell}\frac{(-r)!(r+\ell-1)!}{ \ell!}\zeta(s-\ell)  \notag \\
			&\quad + \sum_{\substack{k=0 \\ k \ne r+\ell-s}}^{r+\ell-1} \frac{(r+\ell-1)!\zeta(r-k)}{(r+\ell-k-1)! x^{k-r-\ell+1}k!} \zeta(s-r-\ell+k+1) + O(|t-(1-r-\ell)|) \label{T1C4}
		\end{align}
	\end{enumerate}
	\end{theorem}
	\begin{proof}
		Let $r \in \mathbb{Z}\setminus \mathbb{N}$ and $\ell \in \mathbb{N} \cup \{0\}$ such that $\ell\geq 1-r$ be fixed and $s \in \mathbb{C}$.
		%		Let $M=r+\ell+ \lfloor|\textup{Re}(s)|\rfloor +1$. 
		Then, for $t>-r-\ell$, \eqref{BOr} and \eqref{BOs} gives, 
		\begin{align}
			y^{t-1} \textup{Li}_s (e^{-xy}) \bigg( \hspace{-0.2 cm} \textup{ Li}_r (e^{-y}) - \Gamma(1-r) y^{r-1} - \sum_{\substack{k=0 }}^{r+\ell+ \lfloor|\textup{Re}(s)|\rfloor +1} (-1)^k\zeta(r-k) \frac{y^k}{k!}\bigg) = O\left(y^{\epsilon-1}\right), \label{BOt}
		\end{align} 
		for some $\e >0$. For $M=r+\ell+ \lfloor|\textup{Re}(s)|\rfloor +1$, from \eqref{BO1} and \eqref{BOt}, the integral on the left-hand side of \eqref{PE} is now defined and analytic in the extended region $\textup{Re}(t)>-r-\ell$  by using \cite[Theorem 2.2]{temme}. Hence, due to the pole of $\Gamma(t)$ at $t=1-r-\ell$, we see
		\begin{align*}
			\lim_{t \to 1-r-\ell} \Bigg( \frac{1}{\Gamma(t)} \int_{0}^{\infty} y^{t-1} \textup{Li}_s (e^{-xy}) \bigg( \hspace{-0.2 cm} \textup{ Li}_r (e^{-y}) - \Gamma(1-r) y^{r-1} - \sum_{\substack{k=0 }}^{r+\ell+ \lfloor|\textup{Re}(s)|\rfloor +1} (-1)^k\zeta(r-k) \frac{y^k}{k!}\bigg) \, dy \Bigg) =0.
			%			&=  \Theta(r,s,t,x) - \frac{\Gamma(1-r)\Gamma(t+r-1)}{ x^{t+r-1}\Gamma(t)}\zeta(t+s+r-1)    - \sum_{\substack{k=0 }}^{M} \frac{ (-1)^k}{k!} \frac{\zeta(r-k)\Gamma(t+k)}{ x^{t+k}\Gamma(t)} \zeta(t+k+s).
		\end{align*} 
		Hence, from \eqref{PE}
		\begin{align}
			\Theta(r,s,t,x) &= \frac{\Gamma(1-r)\Gamma(t+r-1)}{ x^{t+r-1}\Gamma(t)}\zeta(r+s+t-1) \notag \\
			& \quad + \sum_{\substack{k=0}}^{r+\ell+ \lfloor|\textup{Re}(s)|\rfloor +1} \frac{(-1)^k\zeta(r-k)\Gamma(t+k)}{ x^{t+k}k!\Gamma(t)} \zeta(s+t+k) + O(|t-(1-r-\ell)|). \label{M1R}
		\end{align}
		To get the required result for $\Theta(r,s,t,x)$, we analyze the behaviour of the right-hand side of \eqref{M1R} around $t=1-r-\ell$. Since $\Gamma(t)$ has a pole at $t=1-r-\ell$, only the terms with a pole in the numerator survive as $t \to 1-r-\ell$. We handle the remaining calculations in cases. \\ 
		Case I:
Let $s \in \mathbb{Z}$ such that $s>r+\ell$ and $s\ne \ell+1$, or $s \in \mathbb{C} \setminus \mathbb{Z}$. Note that, $\Gamma(t+k)$ has a simple pole only for $k \leq r+\ell-1$. For any $k$, $\zeta(s+t+k)$ has no pole since $s>r+\ell$ or $s \notin \mathbb{Z}$.
		$\Gamma(t+r-1)$ has a simple pole.
		$\zeta(r+s+t-1)$ has no pole since $s \ne \ell+1$.
		%			\vspace{1 cm}
		%			
		%			1. $\Gamma(t+k)$ has a simple pole only for $k \leq r+\ell-1$.\\
		%			2. For any $k$, $\zeta(s+t+k)$ has no pole since $s>r+\ell$ or $s \notin \mathbb{Z}$.\\ 
		%			3. $\Gamma(t+r-1)$ has a simple pole.\\
		%			4. $\zeta(r+s+t-1)$ has no pole since $s \ne \ell+1$.\\
		Use the functional equation $\Gamma(t+1)=t \Gamma(t)$ repeatedly to get,
		\begin{align*}
			\frac{\Gamma(t+k)}{\Gamma(t)}=t (t+1)\cdots(t+k-1). 
		\end{align*}
		Taking limit as $t \to 1-r-\ell$ on both sides of the above equation, for $0\leq k \leq r+\ell-1$ we get,
		\begin{align}
			\lim_{t\to 1-r-\ell} \left( \frac{\Gamma(t+k)}{\Gamma(t)} \right) = (1-r-\ell)(2-r-\ell)\cdots (k-r-\ell) = (-1)^{k} \frac{(r+\ell-1)!}{(r+\ell-k-1)!}. \label{Glim1}
		\end{align}
		Also, since $r \in \mathbb{Z} \setminus \mathbb{N}$, 
		\begin{align*}
			\frac{\Gamma(t+r-1)}{\Gamma(t)} = \frac{\Gamma(t+r-1)}{(t-1)(t-2) \cdots (t+r-1) \Gamma(t+r-1)}  = \frac{1}{(t-1)(t-2) \cdots (t+r-1)}
		\end{align*}
		Hence, we can see,
		\begin{align}
			\lim_{t \to 1-r-\ell} \left( \frac{\Gamma(t+r-1)}{\Gamma(t)}  \right) = (-1)^{r-1} \frac{(r+\ell-1)!}{\ell!}. \label{Glim2}
		\end{align}
		Use \eqref{Glim1} and \eqref{Glim2} in \eqref{M1R} to get the required result \eqref{T1C1}.\\
		Case II: Let $s=\ell+1$. For any $k$, $\zeta(s+t+k)$ has no pole since $r\le0$. Also, $\Gamma(t+r-1)$ has a simple pole, and
		$\zeta(r+s+t-1)$ has a simple pole since $s = \ell+1$. Also, observe that, for any $k, r \in \mathbb{N}$ such that $k \leq r-1$, around $t=1-r$, from Laurent series expansions of the $\Gamma(z)$, we have 
		\begin{align}
			\frac{\Gamma(t+r-1)}{\Gamma(t)}&= (-1)^{r-1}\frac{ (r+\ell-1)!}{\ell!}-(-1)^{r}\frac{ (r+\ell-1)!}{\ell!} (H_{\ell}-H_{r+\ell-1}) (t-(1-r-\ell)) \notag \\ & \qquad + O(|t-(1-r-\ell)|^2). \label{Glim3}
		\end{align}
		Also, the Laurent series of $x^{r+t-1}$ with additional term around $r+t=1-\ell$ is as follows,
		\begin{align}
			\frac{1}{x^{t+r-1}} =x^{\ell} - x^{\ell} \log(x) (t-(1-r-\ell)) + O(|t-(1-r-\ell)|^2). \label{xlim}
		\end{align}
		The well-known Laurent series expansion of $\zeta(z)$,
		\begin{align}
			\zeta(t+r+\ell)= \frac{1}{(t-(1-r-\ell))} + \gamma - \gamma_1 (t-(1-r-\ell)) +O(|t-(1-r-\ell)|^2). \label{zlim}
		\end{align}
		Combining the equations \eqref{Glim1}, \eqref{Glim3}, \eqref{xlim} and \eqref{zlim} in \eqref{M1R}, we get \eqref{T1C2}.\\ 
		Case III: Let $s \in \mathbb{Z} \setminus \mathbb{N}$ such that $s \leq r+\ell$. Observe that $\zeta(s+t+k)$ has a simple pole only for $k=r+\ell-s$, since $s\leq r+\ell$. Also note that $r+\ell-s> r+\ell-1$ since $s \in \mathbb{Z} \setminus  \mathbb{N}$. $\Gamma(t+r-1)$ has a simple pole. $\zeta(r+s+t-1)$ has no pole since $s \leq r+ \ell$. Hence, from \eqref{M1R}, we get \eqref{T1C3}.
		\\ 
	   Case IV: Let $s \in \mathbb{N}$ such that $s \leq r+\ell$. Note that $\zeta(s+t+k)$ has a simple pole only for $k=r+\ell-s$, since $s\leq r+\ell$. Also note that $r+\ell-s\leq r+\ell-1$ since $s \in \mathbb{N}$. $\Gamma(t+r-1)$ has a simple pole. $\zeta(r+s+t-1)$ has no pole since $s \leq r+ \ell$.
		%			\begin{align}
			%				\zeta(t+r+\ell)= \frac{1}{(t-1+r+\ell)} + \gamma - \gamma_1 (t-1+r+\ell) +O((t-1+r+\ell)^2).
			%			\end{align}
	%	\textcolor{red}{Use Laurent series of $x^{t+k}$ and $\Gamma(t+k)/\Gamma(t)$ with additional term for $k=r+\ell-s$ around $r+t=1-\ell$.} 
	Hence, from \eqref{M1R}, we get \eqref{T1C4},which completes the proof.
	\end{proof}
	
	\begin{corollary}\label{corollary0}
	Let $r_0,s_0,t_0$ be three negative integers such that all of them are odd or exactly one of them is odd, then $\Theta(r_0,s_0,t,x ) $ vanishes as $t \to t_0$.
	\end{corollary}
	\begin{proof}
	By taking $\ell=1-r_0-t_0$ in \eqref{T1C3}, around $t=t_0$, we have
	\begin{align*}
		\Theta(r_0,s_0,t,x)
		&= (-1)^{r-1}x^{1-r_0-t_0}\frac{(-r_0)!(-t_0)!}{ (1-r_0-t_0)!}\zeta(r_0+s_0+t_0-1)   + \sum_{\substack{k=0}}^{-t_0} \frac{(-t_0)! \zeta(r_0-k) \zeta(s_0+t_0+k) }{ x^{k+t_0}k!(-t_0-k)!} \notag\\
		&\quad +(-1)^{s_0-1}x^{s_0-1}\frac{(-t_0)!(-s_0)!}{ (1-t_0-s_0)!} \zeta(r_0+s_0+t_0-1)+ O(|t-t_0|). 
	\end{align*} 
	Due to the trivial zeroes of $\zeta(s)$ at negative even integers, from the choices of $r_0,s_0$ and $t_0$, it is clear that the constant term of the above expression equals zero. By tending $t \to t_0$, we see that the proof is complete.
	\end{proof}
	
	\begin{theorem}\label{2ndpolar} \rm 
		\textit{Let $r \in \mathbb{N}$ and $\ell \in \mathbb{N} \cup \{0\}$ be arbitrary such that $\ell \ge 1-r$. Then, the following hold:}
		\begin{enumerate}
			\item  \textit{For a fixed $s \in \mathbb{Z}$ such that $s>r+\ell$ , or $s \in \mathbb{C} \setminus \mathbb{Z}$, we have,}
			\begin{align}
			\Theta(r,s,t,x) 
			&= \binom{r+\ell-1}{\ell} x^{\ell} \frac{\zeta(s-\ell)}{(t-(1-r-\ell))} + \binom{r+\ell-1}{\ell} x^{\ell} \zeta(s-\ell) \left(\gamma+H_{r-1}-H_{r+\ell-1}\right) \notag  \\
			&\quad +\sum_{\substack{k=0 \\ k\ne r-1}}^{r+\ell-1} \binom{r+\ell-1}{k} x^{r+\ell-1-k} \zeta(r-k) \zeta(1+k-\ell-r+s) + O(|t-(1-r-\ell)|). \label{T2C1}
		\end{align}
		\\
		\item   \textit{For a fixed $s \in \mathbb{Z} \setminus \mathbb{N}$ such that $s \leq r+\ell$, we have,}
		\begin{align}
			\Theta(r,s,t,x) &=	 \binom{r+\ell-1}{\ell} x^{\ell} \frac{\zeta(s-\ell)}{(t-(1-r-\ell))} + \binom{r+\ell-1}{\ell} x^{\ell} \zeta(s-\ell) \left(\gamma+H_{r-1}-H_{r+\ell-1}\right) \notag  \\
			& \quad + \sum_{\substack{k=0 \\ k\ne r-1}}^{r+\ell-1} \binom{r+\ell-1}{k} x^{r+\ell-1-k} \zeta(r-k) \zeta(1+k-\ell-r+s) \notag  \\
			& \quad  + (-1)^{s-1} \frac{(r+\ell-1)!(-s)!}{(r+\ell-s)!} x^{s-1} \zeta(s-\ell) + O(|t-(1-r-\ell)|). \label{T2C2}
		\end{align}
		\\
		\item   \textit{For a fixed $s \in \mathbb{N}$ such that $s\leq r+\ell$ and $s \ne \ell+1$, we have,}
			\begin{align}
			\Theta(r,s,t,x) &= \binom{r+\ell-1}{\ell} x^{\ell} \frac{\zeta(s-\ell)}{(t-(1-r-\ell))} + \binom{r+\ell-1}{\ell} x^{\ell} \zeta(s-\ell) \left(\gamma+H_{r-1}-H_{r+\ell-1}\right) \notag  \\
			&\quad + \binom{r+\ell-1}{s-1}x^{s-1} \frac{\zeta(s-\ell)}{(t-(1-r-\ell))} + \binom{r+\ell-1}{s-1}x^{s-1} \zeta(s-\ell) \left( \gamma + H_{s-1} -H_{r+\ell-1} - \log(x) \right) \notag \\
			&\quad +\sum_{\substack{k=0 \\ k\ne r-1 \\ k\ne r+\ell-s}}^{r+\ell-1} \binom{r+\ell-1}{k} x^{r+\ell-1-k} \zeta(r-k) \zeta(1+k-\ell-r+s)  + O(|t-(1-r-\ell)|). \label{T2C3}
			\end{align}
		\\
		\item   \textit{For $s = \ell+1$, we have}
		\begin{align}
			\Theta(r,s,t,x) & =  \binom{r+\ell-1}{\ell}  \frac{ 2 x^{\ell}}{(t-(1-r-\ell))^2} + \binom{r+\ell-1}{\ell}  \frac{  x^{\ell} \left(2 \gamma+ H_{\ell}+H_{r-1}-2 H_{r+\ell-1}-\log(x)\right)}{(t-(1-r-\ell))} \notag \\
			&\quad + \frac{1}{3}\binom{r+\ell-1}{\ell} x^{\ell} \bigg( \hspace{-0.2 cm} -\pi^2+ 3 (\gamma+ H_{\ell}-H_{r+\ell-1}-\log(x)) (\gamma+H_{r-1}-H_{r+\ell-1})+3 \psi'(r+\ell)\bigg) \notag \\
			&\quad 	+ \sum_{\substack{k=0 \\ k\ne r-1}}^{r+\ell-1} \binom{r+\ell-1}{k} x^{r+\ell-1-k} \zeta(r-k) \zeta(k-r+2) + O(|t-(1-r-\ell)|). \label{T2C4}
		\end{align}		
	\end{enumerate}
	\end{theorem}
	\begin{proof}
		Let $s \in \mathbb{C}$. Then, for $t>-r-\ell$, \eqref{rLim} and \eqref{BOt} gives,  
		\begin{align}
			\textup{ Li}_r (e^{-y}) - (-1)^{r-1} (H_{r-1} -\log(y)) \frac{y^{r-1}}{(r-1)!} - \sum_{\substack{k=0 \\ k\ne r-1}}^{ r+\ell+ \lfloor|\textup{Re}(s)|\rfloor +1} (-1)^k\zeta(r-k) \frac{y^k}{k!} = O(y^{\epsilon -1}) \label{BOt1}
		\end{align} 
		for some $\e >0$. As in the previous theorem, for $M=r+\ell+ \lfloor|\textup{Re}(s)|\rfloor +1$, from \eqref{BO1} and \eqref{BOt1}, the integral on the left-hand side of \eqref{CE} is now defined and analytic in the extended region $\textup{Re}(t)>-r-\ell$  by using \cite[Theorem 2.2]{temme}. Hence, due to the pole of $\Gamma(t)$ at $t=1-r-\ell$, we see
		\begin{align*}
			\lim_{t \to 1-r-\ell} \Bigg( \frac{1}{\Gamma(t)} \int_{0}^{\infty} y^{t-1} \textup{Li}_s (e^{-xy}) \bigg( \textup{ Li}_r(e^{-y}) -& (-1)^{r-1} (H_{r-1} -\log(y)) \frac{y^{r-1}}{(r-1)!} \\ &- \sum_{\substack{k=0 \\ k\ne r-1}}^{r+\ell+ \lfloor|\textup{Re}(s)|\rfloor +1} (-1)^k\zeta(r-k) \frac{y^k}{k!} \bigg) \, dy \Bigg) =0.
			%			&=  \Theta(r,s,t,x) - \frac{\Gamma(1-r)\Gamma(t+r-1)}{ x^{t+r-1}\Gamma(t)}\zeta(t+s+r-1)    - \sum_{\substack{k=0 }}^{M} \frac{ (-1)^k}{k!} \frac{\zeta(r-k)\Gamma(t+k)}{ x^{t+k}\Gamma(t)} \zeta(t+k+s).
		\end{align*}
		Hence, from \eqref{CE},
		\begin{align}
			\Theta(r,s,t,x) & =  \frac{(-1)^{r-1}\Gamma(t+r-1)}{x^{t+r-1} (r-1)!\Gamma(t)}\Big(\zeta(t+s+r-1) \left( H_{r-1}  - \psi(t+r-1) + \log(x) \right)- \zeta'(t+s+r-1)\Big) \notag \\ 
			&\quad +\sum_{\substack{k=0 \\ k\ne r-1}}^{r+\ell+ \lfloor|\textup{Re}(s)|\rfloor +1} \frac{ (-1)^k}{k!} \frac{\zeta(r-k)\Gamma(t+k)}{ x^{t+k}\Gamma(t)} \zeta(t+k+s) + O(t-(1-r-\ell)). \label{M1N}
		\end{align}
		To get the required Kronecker limit formula for $\Theta(r,s,t,x)$, we analyze the behaviour of the right-hand side of \eqref{M1R} around $t=1-r-\ell$. Since $\Gamma(t)$ has a pole at $t=1-r-\ell$, only the terms with a pole in the numerator survive, as $t \to 1-r-\ell$. We handle remaining calculations in cases. \\ 
		Case I: Let $s \in \mathbb{Z}$ such that $s>r+\ell$ , or $s \in \mathbb{C} \setminus \mathbb{Z}$. Note that, $\Gamma(t+k)$ has a simple pole only for $k \leq r+\ell-1$. Also, $\zeta(s+t+k)$ and $\zeta(r+s+t-1)$ have no poles since $s > r+\ell$ or $s \notin \mathbb{Z}$.  $\Gamma(t+r-1)$ has a simple pole.  $\psi(t+r-1)$ has a simple pole. From \eqref{M1N}, using \eqref{Glim1} and \eqref{Glim3}, we get \eqref{T2C1}. \\ 
		Case II: Let $s \in \mathbb{Z} \setminus \mathbb{N}$ such that $s \leq r+\ell$. Observe that $\zeta(s+t+k)$ has a simple pole for $k =r+\ell-s$, moreover $r+\ell-s > r+\ell-1$. $\zeta(r+s+t-1)$ has no pole since $s \ne \ell+1$.  $\Gamma(t+r-1)$ has a simple pole. $\psi(t+r-1)$ has a simple pole. Also observe that the term corresponding to $k=r+\ell-s$ appears due to the pole of $\zeta(t+k+s)$ and not $\Gamma(t+k)$. From \eqref{M1N}, we get \eqref{T2C2}. \\  
		Case III: Let  $s \in \mathbb{N}$ such that $s\leq r+\ell$ and $s \ne \ell+1$. Then $\zeta(s+t+k)$ has a simple pole for $k =r+\ell-s$, moreover $r+\ell-s \le r+\ell-1$. $\zeta(r+s+t-1)$ has no pole since $s \ne \ell+1$.  $\Gamma(t+r-1)$ has a simple pole. $\psi(t+r-1)$ has a simple pole. As in the previous case, the term corresponding to $k=r+\ell-s$ appears due to the pole of $\zeta(t+k+s)$ and not $\Gamma(t+k)$. From \eqref{M1N}, we get \eqref{T2C3}. \\  
		Case IV: Let $s=\ell+1$. Then $\zeta(s+t+k)$ has no pole since $s=\ell+1$ and $ k \ne r-1$. $\zeta(r+s+t-1)$ and $\zeta'(r+s+t-1)$ have a  simple and a double pole respectively since $s = \ell+1$.  $\Gamma(t+r-1)$ has a simple pole. $\psi(t+r-1)$ has a simple pole. Also, use the following power series, around $t=1-r-\ell$,
		\begin{align}
		&\frac{\Gamma(t+r-1)}{\Gamma(t)}= (-1)^{r-1}\frac{ (r+\ell-1)!}{\ell!}-(-1)^{r}\frac{ (r+\ell-1)!}{\ell!} (\psi(\ell+1)-\psi(r+\ell)) (t-(1-r-\ell))  \notag \\
		& \qquad -(-1)^{r}\frac{ (r+\ell-1)!}{2\ell!} \left( (\psi(\ell+1)-\psi(r+\ell))^2 - \psi'(\ell+1)+\psi'(r+\ell)  \right)(t-(1-r-\ell))^2 \notag \\& \qquad + O(|(t-(1-r-\ell))|^3), \label{Glim4}
		\end{align}
		\begin{align}\label{xlim5}
		\frac{1}{x^{t+r-1}} = x^{\ell}-x^{\ell} \log(x) (t-(1-r-\ell)) + \frac{1}{2} x^{\ell} \log^2(x) (t-(1-r-\ell))^2 +O(|(t-(1-r-\ell))|^3).
		\end{align}
		Use \eqref{Glim4} and \eqref{xlim5} in \eqref{M1N} to get the required result \eqref{T2C4}. 
	\end{proof}
%	We now prove a similar Kronecker Limit type formula at a \textcolor{red}{point of singularity of type 2}. 	
	Next, we are interested to get Kronecker limit type formula for $\Theta(r,s,t,x)$ in the \emph{third variable} $t$, around $t= 2-r-s$. We can obtain such a result for fixed $r,s \in \mathbb{C}$ such that $r+s \in \mathbb{N}$ and $r+s \ge 2$ in the following way: 
	\begin{itemize}
		\item $r, s \in \mathbb{N}$: It reduces to case IV of Theorem \ref{2ndpolar} upon taking $s=\ell+1$.
		\item $s \in \mathbb{N}$ and $r \in \mathbb{Z} \setminus \mathbb{N}$: It reduces to case II of Theorem \ref{1stpolar} upon taking $s=\ell+1$.
		\item $r \in \mathbb{N}$ and $s \in \mathbb{Z} \setminus \mathbb{N}$: It reduces to the previous case, $s \in \mathbb{N}$ and $r \in \mathbb{Z} \setminus \mathbb{N}$, after using \eqref{Tinv}.
	\end{itemize}
	Hence, for a fixed $r,s \in  \mathbb{C} \setminus \mathbb{Z}$, we state the formula as the following theorem.  
	
	\begin{theorem}\label{3rdpolar}
		Let  $r,s \in  \mathbb{C} \setminus \mathbb{Z}$ such that $r+s \in \mathbb{N}$ and $r+s \ge 2$. Around $t=2-r-s$, we have
		\begin{align}
			\Theta(r,s,t,x) & = (-1)^{r+s} x^{s-1} (r+s-2)! \Gamma(1-r) \Gamma(1-s) \notag \\
			&\quad + \sum_{k=0}^{r+s-2} \binom{r+s-2}{k} x^{r+s-2-k}  \zeta(r-k) \zeta(k-r+2) + O\left( |t-(2-r-s)| \right). \label{T3}
		\end{align}
	\end{theorem}
	\begin{proof}
		Then, for $\textup{Re}(t)>1-r-s$, \eqref{BOr} and \eqref{BOs} gives, 
		\begin{align}
			y^{t-1} \textup{Li}_s (e^{-xy}) \bigg( \hspace{-0.2 cm} \textup{ Li}_r (e^{-y}) - \Gamma(1-r) y^{r-1} - \sum_{\substack{k=0 }}^{r+ s+\lfloor|\textup{Re}(s)|\rfloor } (-1)^k\zeta(r-k) \frac{y^k}{k!}\bigg) = O\left(y^{\epsilon-1}\right), \label{BOt2}
		\end{align} 
		for some $\e >0$. For $M=r+s+\lfloor|\textup{Re}(s)|\rfloor $, from \eqref{BO1} and \eqref{BOt2}, the integral on the left-hand side of \eqref{PE} is now defined and analytic in the extended region $\textup{Re}(t)>1-r-s$  by using \cite[Theorem 2.2]{temme}. Hence, due to the pole of $\Gamma(t)$ at $t=2-r-s$, we see
		\begin{align*}
			\lim_{t \to 2-r-s} \Bigg( \frac{1}{\Gamma(t)} \int_{0}^{\infty} y^{t-1} \textup{Li}_s (e^{-xy}) \bigg( \hspace{-0.2 cm} \textup{ Li}_r (e^{-y}) - \Gamma(1-r) y^{r-1} - \sum_{\substack{k=0 }}^{r+s+ \lfloor|\textup{Re}(s)|\rfloor } (-1)^k\zeta(r-k) \frac{y^k}{k!}\bigg) \, dy \Bigg) =0.
			%			&=  \Theta(r,s,t,x) - \frac{\Gamma(1-r)\Gamma(t+r-1)}{ x^{t+r-1}\Gamma(t)}\zeta(t+s+r-1)    - \sum_{\substack{k=0 }}^{M} \frac{ (-1)^k}{k!} \frac{\zeta(r-k)\Gamma(t+k)}{ x^{t+k}\Gamma(t)} \zeta(t+k+s).
		\end{align*} 
		Hence, from \eqref{PE}
		\begin{align}
			\Theta(r,s,t,x) &= \frac{\Gamma(1-r)\Gamma(t+r-1)}{ x^{t+r-1}\Gamma(t)}\zeta(r+s+t-1) \notag \\
			& \quad + \sum_{\substack{k=0}}^{r+s+ \lfloor|\textup{Re}(s)|\rfloor } \frac{(-1)^k\zeta(r-k)\Gamma(t+k)}{ x^{t+k}k!\Gamma(t)} \zeta(s+t+k)  + O\left( |t-(2-r-s)| \right). \label{M1C}
		\end{align}
		Since $\Gamma(t)$ has a pole at $t=2-r-s$, only the terms with a pole in the numerator survive, as $t \to 2-r-s$. We list the poles of numerator at $t=2-r-s$: Note that, $\Gamma(t+k)$ has a simple pole only for $k \leq r+s-2$. For any $k$, $\zeta(s+t+k)$ has no pole since $r \notin \mathbb{Z}$ . $\Gamma(t+r-1)$ has no pole since $s \notin \mathbb{Z}$. $\zeta(r+s+t-1)$ has a simple pole. Hence, use the laurent series expansions of the functions on the right-hand side of \eqref{M1C} to get the required result \eqref{T3}.
	\end{proof}

	\section{Kronecker limit type formula for $\Theta(r,s,t,x)$ in the second variable}\label{Section4}
	The method used in Section \ref{Sec2} cannot be used to find Kronecker limit type formula for $\Theta(r,s,t,x)$ in the \textit{second variable} $s$. Hence, we come with an alternative method to obtain the same, which involves the series evaluations of $\Theta(r,s,t,x)$ in terms of Herglotz-Zagier type functions \cite{zagier}. The evaluation is obtained by the partial fraction method, which is very effective in the theory of multiple zeta functions. For instance, Gangl, Kaneko and Zagier \cite{partialfraction} used the partial fraction, for integers $i,j \geq 2$, 
	\begin{align*}
		\frac{1}{m^i n^j} =  \sum_{a+b=i+j} \left( \binom{a-1}{i-1} \frac{1}{(m+n)^an^b}    + \binom{a-1}{j-1}  \frac{1}{(m+n)^an^b}   \right),
	\end{align*}
	to prove the Euler decomposition formula, 
	\begin{align*}
		\zeta(i) \zeta(j) = \sum_{a+b=i+j} \left( \binom{a-1}{i-1} \zeta(a,b)  +  \binom{a-1}{j-1} \zeta(a,b)  \right),
	\end{align*} 
	where $\zeta(s_1,s_2)$ is the \textit{double zeta function}. Guo and Xie \cite{guoxie} used the partial fraction method to obtain natural shuffle algebra structure for the integral representations of the multiple-zeta values. We prove the partial fraction in the following Lemma \ref{parfra} before using it to obtain the series evaluation.
		\begin{lemma}\label{parfra}
		Fix $r, t \in\mathbb{N} \cup \{0\}$ such that $r+t\geq1$. Then, for $n,y>0$, the following partial fraction holds,
		%		\textcolor{red}{check conditions on $n,y,r,t$}
		\begin{align}
			\frac{1}{n^r (n+y)^t} = (-1)^r \sum_{j=0}^{t-1} \binom{j+r-1}{j} \frac{1}{y^{j+r} (n+y)^{t-j}} + \sum_{i=0}^{r-1} \binom{i+t-1}{i} \frac{(-1)^i}{n^{r-i} y^{t+i}}. \label{parts1}
		\end{align}
		with the convention that $\binom{-1}{0}=1$, $\binom{c}{-1}=0$ for any $c \in \mathbb{Z}\backslash \{-1\}$ with an exception $\binom{-1}{-1} =1$.
	\end{lemma}
	\begin{proof}
		Start with the expression on the right-hand side of \eqref{parts1} and take least common denominator to get,
		\begin{align}
			&(-1)^r \sum_{j=0}^{t-1} \binom{j+r-1}{j} \frac{y^{t-j}(n+y)^{j}}{y^{r+t} (n+y)^{t}} + \sum_{i=0}^{r-1} \binom{i+t-1}{i} \frac{(-1)^i n^i y^{r-i}}{n^{r} y^{r+t}} \notag \\
			&=(-1)^r \sum_{j=0}^{t-1} \binom{j+r-1}{j} \frac{n^{r} y^{t-j}(n+y)^{j}}{n^{r} y^{r+t} (n+y)^{t}} + \sum_{i=0}^{r-1} \binom{i+t-1}{i} \frac{(-1)^i n^i y^{r-i} (n+y)^{t}}{n^{r} y^{r+t} (n+y)^{t}} \notag \\
%			&= \frac{1}{n^r y^{r+t} (n+y)^t} \left( (-1)^r \sum_{j=0}^{t-1} \binom{j+r-1}{j} n^{r} y^{t-j}(n+y)^{j} + \sum_{i=0}^{r-1} \binom{i+t-1}{i} (-1)^i n^i y^{r-i} (n+y)^{t}\right) \notag \\
			%			&\textup{\textcolor{red}{where $y^{-1}$ has been cancelled from the numerator and denominator}} \notag \\
%			&= \frac{1}{n^r y^{r+t} (n+y)^t} \left( (-1)^r  n^{r} \sum_{j=0}^{t-1} \binom{j+r-1}{j} y^{t-j}(n+y)^{j} + (n+y)^{t} \sum_{i=0}^{r-1} \binom{i+t-1}{i} (-1)^i n^i y^{r-i} \right) \notag \\
			&= \frac{A(n,y,r,t)}{n^r y^{r+t} (n+y)^t}, \label{calc}
		\end{align}
		where 
		\begin{align*}
			A(n,y,r,t):= (-1)^r  n^{r} \sum_{j=0}^{t-1} \binom{j+r-1}{j} y^{t-j}(n+y)^{j} + (n+y)^{t} \sum_{i=0}^{r-1} \binom{i+t-1}{i} (-1)^i n^i y^{r-i}.
		\end{align*}
		Clearly, for a fixed $r,t \in \mathbb{N}$, $A(n,y,r,t)$ is a polynomial of two variables $n,y$.
		\begin{align}
			A(n,y,r,t)&= (-1)^r  n^{r} \sum_{j=0}^{t-1} \binom{j+r-1}{j} y^{t-j} \sum_{c=0}^{j} \binom{j}{c} n^c y^{j-c}  + \sum_{d=0}^{t} \binom{t}{d} n^d y^{t-d} \sum_{i=0}^{r-1} \binom{i+t-1}{i} (-1)^i n^i y^{r-i} \notag \\
			&=  (-1)^r   \sum_{j=0}^{t-1} \sum_{c=0}^{j} \binom{j+r-1}{j}  \binom{j}{c} n^{r+c}   y^{t-c}   + \sum_{d=0}^{t} \sum_{i=0}^{r-1}  \binom{i+t-1}{i} \binom{t}{d} (-1)^i n^{d+i} y^{t-d+r-i} \notag \\
			&=  (-1)^r  \sum_{c=0}^{t-1} n^{r+c}  \sum_{j=c}^{t-1}  \binom{j+r-1}{j}  \binom{j}{c}    y^{t-c}   + \sum_{d=0}^{t} \sum_{i=0}^{r-1}  \binom{i+t-1}{i} \binom{t}{d} (-1)^i n^{d+i} y^{t-d+r-i}.   \label{last-n}
		\end{align}
		$A(n,y,r,t)$ can be seen as a polynomial in $n$ with coefficients as polynomials in $y$. Clearly, the highest degree of $n$ possible in $A(n,y,r,t)$ is $r+t-1$, where $r,t \in \mathbb{N}$ are fixed. We now show explicitly evaluate the coefficients of $n^a$ for $1\leq a\leq r+t-1$ in two cases.\\
%		\vspace{0.2 cm}
		%		\textcolor{red}{Understand coefficients of $n^{a}$ for all $a \in \mathbb{N}$.}	
		Case I: For $1\leq a\leq r-1$.\\
		Observe that the first term on the right-hand side of \eqref{last-n} does not contribute to the coefficient of $n^a$ since least power of $n$ possible is $r$. From the second term, the only way to get $n^a$ is when $d\leq a$ and $i=a-d$. Hence we get,
		\begin{align}
			\textup{Coefficient of }n^a \textup{ in } A(n,y,r,t) & \hspace{0.23 cm}= \sum_{d=0}^{a} \binom{a-d+t-1}{a-d} \binom{t}{d} (-1)^{a-d}y^{t+r-a} \notag \\
			&  \hspace{0.23 cm}= (-1)^{a} y^{t+r-a} \sum_{d=0}^{a} \binom{a-d+t-1}{a-d} \binom{t}{d} (-1)^d  \notag  \\
			&\buildrel \rm \emph{d} \rightarrow \emph{a}-\emph{d} \over = y^{t+r-a} \sum_{d=0}^{a} \binom{d+t-1}{d} \binom{t}{a-d} (-1)^d   \notag \\
			&\hspace{0.23 cm}=   y^{t+r-a} \sum_{d=0}^{a} \frac{(d+t-1)!}{d!(t-1)!} \cdot \frac{t!}{(a-d)!(d+t-a)!} \cdot (-1)^d   \notag  \\
			&\hspace{0.23 cm}=  \sum_{d=0}^{a} \frac{(d+t-1)!}{d!(a-d)!(d+t-a)!} (-1)^d  \notag  \\
			& \hspace{0.23 cm}= \frac{t y^{t+r-a}}{a!} \sum_{d=0}^{a} (-1)^d \binom{a}{d} p(d), \label{p(d)1}
		\end{align}
		where $p(d)= \frac{(d+t-1)!}{(d+t-a)!} = (d+t-1)(d+t-2) \cdots (d+t-a+1)$ is a polynomial in $d$ of degree $a-1$. As explained in \cite[Section 3]{Spivey} and \cite[Corollary 2]{ruiz}, the sum on the right-hand side of \eqref{p(d)1} is zero, since the polynomial $p(d)$ has a degree smaller than $a$. Hence, the coefficient of $n^a$ in $A(n,y,r,t)$ is zero for $1\leq a\leq r-1$.\\
%		\vspace{0.2 cm}
		Case II: For $r\leq a\leq r+t-1$. \\
		Observe that the term $c=a-r$ in the first term of the right-hand side of \eqref{last-n} alone contributes to the coefficient of $n^a$. The condition $ a\leq r+t-1$ ensures that $a-r \leq t-1$. From the second term of the right-hand side of \eqref{last-n}, for any $i$, $d=a-i$ alone contributes to the coefficient of $n^a$. Since  $r\leq a\leq r+t-1$, we have $d=a-i\leq t$ for any $i$. Since $i\leq r-1$, we always have $d\geq0$. Hence we get,
		\begin{align}
			\textup{Coefficient of }n^a \textup{ in } A(n,y,r,t) &= y^{t+r-a}\bigg(	A_1(n,y,r,t)+A_2(n,y,r,t)\bigg). \label{AA1A2}
		\end{align}
		where,
		\begin{align}
			A_1(n,y,r,t)&:= (-1)^r  \sum_{j=a-r}^{t-1} \binom{j+r-1}{j} \binom{j}{a-r}, \label{A1def} \\  
			A_2(n,y,r,t)&:= \sum_{i=0}^{r-1} \binom{i+t-1}{i} \binom{t}{a-i} (-1)^i. \label{A2def}
		\end{align}
		We now evaluate both $	A_1(n,y,r,t)$ and $	A_2(n,y,r,t)$ explicitly. Firstly, after expanding the binomials in $A_1(n,y,r,t)$, we get,
		\begin{align}
				A_1(n,y,r,t) &= \frac{(-1)^r}{(r-1)!(a-r)!} \sum_{j=a-r}^{t-1} \frac{(j+r-1)!}{(j+r-a)!} \notag \\
				&\hspace{-0.44 cm}\buildrel \rm \emph{j} \rightarrow \emph{i}+\emph{a}-\emph{r} \over = \frac{(-1)^r}{(r-1)!(a-r)!} \sum_{i=0}^{r+t-a-1} \frac{(i+a-1)!}{i!} \notag \\
				&= \frac{(-1)^r(a-1)!}{(r-1)!(a-r)!} \sum_{i=0}^{r+t-a-1} \binom{i+a-1}{i} \notag \\
				&= \frac{(-1)^r(a-1)!}{(r-1)!(a-r)!} \binom{r+t-1}{r+t-a-1} \notag  \\
				&=\frac{(-1)^r}{a} \times \frac{(r+t-1)!}{(r-1)!(t-1)!} \times \frac{(t-1)!}{(a-r)!(r+t-a-1)!} \notag \\
				&= \frac{(-1)^r (r+t-a)}{a} \binom{r+t-1}{t} \binom{t}{a-r}. \label{A1Fin}
		\end{align}
		Next, we evaluate $A_2(n,y,r,t)$. Since $\binom{t}{a-i}=0$ for $i<a-t$, it is enough to consider the sum between $a-t\leq i \leq r-1$, and on expanding the binomials, we get,
		\begin{align}
			A_2(n,y,r,t) &=t \sum_{i=a-t}^{r-1} \frac{(-1)^i}{i! (a-i)!} \times \frac{(i+t-1)!}{(i+t-a)!} \notag \\
			&\hspace{-0.44 cm}\buildrel \rm \emph{i} \rightarrow \emph{j}+\emph{a}-\emph{t} \over = t \sum_{j=0}^{r+t-a-1} \frac{(-1)^{j+a-t}}{(j+a-t)!(t-j)!} \times \frac{(j+a-1)!}{j!} \notag \\
			&= \frac{(-1)^{a+t}}{(t-1)!} \sum_{j=0}^{r+t-a-1} (-1)^j \binom{t}{j} \frac{(j+a-1)!}{(j+a-t)!} \notag \\
			&=  \frac{(-1)^{a+t}}{(t-1)!}  \times \frac{(-1)^{a+r+t}(a-r-t)}{at (r-1)!} \binom{t}{r+t-a} (r+t-1)! \notag \\
			&= - \frac{(-1)^r(r+t-a)}{a} \binom{r+t-1}{t} \binom{t}{a-r}, \label{A2Fin}
		\end{align}
		where we have used the fact $\binom{c}{d} = \binom{c}{c-d}$ in the last step and the following identity obtained from \textit{Mathematica 11} in the second last step,
		\begin{align*}
			\sum_{j=0}^{r+t-a-1} (-1)^j \binom{t}{j} \frac{(j+a-1)!}{(j+a-t)!}=\frac{(-1)^{a+r+t}(a-r-t)}{at (r-1)!} \binom{t}{r+t-a} (r+t-1)!.
		\end{align*}
		On substituting \eqref{A1Fin} and \eqref{A2Fin} in \eqref{AA1A2}, we can see that the coefficient of $n^a$ is zero for $r\leq a \leq r+t-1$.	Thus, both the cases combined, we have shown that the coefficient of $n^a$ is zero for all $1\leq a\leq r+t-1$, and thus, $A(n,y,r,t)$ is independent of $n$, and hence, 
		\begin{align*}
			A(n,y,r,t) = A(0,y,r,t)=y^{r+t}.
		\end{align*}
		Put this in \eqref{calc} to get the required result.
	\end{proof}	
	We now use Lemma \ref{parfra} to prove the follwing theorem for the series evaluations of $\Theta(r,s,t,x)$.
%	\textcolor{red}{The higher evaluations of $\Theta(r,s,t,x)$ is as follows:}
		\begin{theorem}\label{THM-EVAL}
		For $r,t \in \mathbb{N} \cup \{0\}$ such that $r+t\geq1$, and $s \in \mathbb{C}$ such that $\textup{Re}(s)>2-r-t$, we have, 
		\begin{align}
			\Theta(r,s,t,x)  &=  \sum_{i=0}^{r-2} \frac{(-1)^i}{x^{t+i}}  \binom{i+t-1}{i}  \zeta(s+t+i) \zeta(r-i) - \frac{(-1)^r}{x^{r+t-1}} \binom{r+t-2}{t-1}  \sum_{m=1}^{\infty} \frac{\gamma + \psi(mx+1)}{m^{r+s+t-1}} \notag \\
			&\quad + (-1)^r  \sum_{j=0}^{t-2} \binom{j+r-1}{j}  \frac{(-1)^{t-j}}{(t-j-1)!} \frac{1}{x^{j+r}} \sum_{m=1}^{\infty} \frac{\psi^{(t-j-1)}(mx+1)}{m^{r+s+j}}, \label{T-H-EV}
		\end{align}
		with the convention that $\binom{-1}{0}=1$, $\binom{c}{-1}=0$ for any $c \in \mathbb{Z}\backslash \{-1\}$ with an exception $\binom{-1}{-1} =1$.
	\end{theorem}
	\begin{proof}
		From the definition \eqref{Tdef}, we can see 
		\begin{align}
			\Theta(r,s,t,x)&= \sum_{n=1}^{\infty} \sum_{m=1}^{\infty} \frac{1}{n^r m^s (n+mx)^t} = \sum_{m=1}^{\infty} \frac{1}{m^s} \sum_{n=1}^{\infty} \frac{1}{n^r (n+mx)^t}. \label{EvStart}
		\end{align}
		Use Lemma \ref{parfra} with $y=mx$ in \eqref{EvStart} to get,
		\begin{align*}
			\Theta(r,s,t,x) &= \sum_{m=1}^{\infty} \frac{1}{m^s} \sum_{n=1}^{\infty} \left( (-1)^r \sum_{j=0}^{t-1} \binom{j+r-1}{j} \frac{1}{(mx)^{j+r} (n+mx)^{t-j}} + \sum_{i=0}^{r-1} \binom{i+t-1}{i} \frac{(-1)^i}{n^{r-i} (mx)^{t+i}} \right) \\
			&= \sum_{m=1}^{\infty} \frac{1}{m^s} \sum_{n=1}^{\infty} \left( (-1)^r \sum_{j=0}^{t-2} \binom{j+r-1}{j} \frac{1}{(mx)^{j+r} (n+mx)^{t-j}} + \sum_{i=0}^{r-2} \binom{i+t-1}{i} \frac{(-1)^i}{n^{r-i} (mx)^{t+i}} \right) \\
			&\quad+  \sum_{m=1}^{\infty} \frac{1}{m^s} \sum_{n=1}^{\infty} \left( (-1)^r \binom{r+t-2}{t-1} \frac{1}{(mx)^{r+t-1}(n+mx)} + (-1)^{r-1} \binom{r+t-2}{r-1} \frac{1}{n (mx)^{r+t-1}}  \right) \\
			&= \sum_{m=1}^{\infty} \frac{1}{m^s} \sum_{n=1}^{\infty} \left( (-1)^r \sum_{j=0}^{t-2} \binom{j+r-1}{j} \frac{1}{(mx)^{j+r} (n+mx)^{t-j}} + \sum_{i=0}^{r-2} \binom{i+t-1}{i} \frac{(-1)^i}{n^{r-i} (mx)^{t+i}} \right) \\
			&\quad - (-1)^r \binom{r+t-2}{t-1} \frac{1}{x^{r+t-1}} \sum_{m=1}^{\infty} \frac{1}{m^{r+s+t-1}} \sum_{n=1}^{\infty} \left( \frac{1}{n} - \frac{1}{n+mx}  \right),
		\end{align*}
		where we have used $\binom{r+t-2}{r-1} = \binom{r+t-2}{t-1}$. From the series definition of the digamma function $\psi(x)$, we have,
		\begin{align*}
			\Theta(r,s,t,x) &= \sum_{m=1}^{\infty} \frac{1}{m^s} \sum_{n=1}^{\infty} \left( (-1)^r \sum_{j=0}^{t-2} \binom{j+r-1}{j} \frac{1}{(mx)^{j+r} (n+mx)^{t-j}} + \sum_{i=0}^{r-2} \binom{i+t-1}{i} \frac{(-1)^i}{n^{r-i} (mx)^{t+i}} \right) \\
			&\quad-(-1)^r \binom{r+t-2}{t-1} \frac{1}{x^{r+t-1}} \sum_{m=1}^{\infty} \frac{\gamma + \psi(mx+1)}{m^{r+s+t-1}}.
		\end{align*} 
		From \cite[Corollary 2]{coffey}, we have, for $j \in \mathbb{N}, j \geq 2$
		\begin{align}
			\frac{(-1)^j}{(j-1)!}\psi^{(j-1)}(z+1) = \sum_{\ell=1}^{\infty} \frac{1}{(\ell+z)^{j}}. \label{psidash}
		\end{align}
		Hence, we can simplify $\Theta(r,s,t,x)$ as
		\begin{align*}
			\Theta(r,s,t,x) &= (-1)^r  \sum_{j=0}^{t-2} \binom{j+r-1}{j} \frac{1}{x^{j+r}} \sum_{m=1}^{\infty} \frac{1}{m^{r+s+j}} \sum_{n=1}^{\infty} \frac{1}{(n+mx)^{t-j}} \\ &\quad +  \sum_{i=0}^{r-2} \frac{(-1)^i}{x^{t+i}} \binom{i+t-1}{i}  \sum_{m=1}^{\infty} \frac{1}{m^{s+t+i}} \sum_{n=1}^{\infty}  \frac{1}{n^{r-i}} -  \frac{(-1)^r}{x^{r+t-1}} \binom{r+t-2}{t-1} \sum_{m=1}^{\infty} \frac{\gamma + \psi(mx+1)}{m^{r+s+t-1}}.
		\end{align*}
		Use \eqref{psidash} to get,
		\begin{align*}
			\Theta(r,s,t,x) &= (-1)^r  \sum_{j=0}^{t-2} \binom{j+r-1}{j}  \frac{(-1)^{t-j}}{(t-j-1)!} \frac{1}{x^{j+r}} \sum_{m=1}^{\infty} \frac{\psi^{(t-j-1)}(mx+1)}{m^{r+s+j}}   \\
			&\quad +  \sum_{i=0}^{r-2} \frac{(-1)^i}{x^{t+i}}  \binom{i+t-1}{i}  \zeta(s+t+i) \zeta(r-i) - \frac{(-1)^r}{x^{r+t-1}} \binom{r+t-2}{t-1}  \sum_{m=1}^{\infty} \frac{\gamma + \psi(mx+1)}{m^{r+s+t-1}}.
		\end{align*}
		This is the required result \eqref{T-H-EV}.
	\end{proof}
	\subsection{Mixed functional equations} The two-term functional equation given by Guinand \cite{apg3} has higher derivatives of $\psi(x)$ in the summand while the one given by Vlasenko--Zagier \cite[Proposition 4]{vz} has the higher power of $m$ in the denominator. We call the functional equation obtained in \eqref{mixed}, the mixed functional equations since they contain both, a series containing higher derivatives of $\psi(x)$ and containing higher powers of $m$ in the denominator each.
	\begin{remark}\label{REMARK1}
		The identity \eqref{newIDex} holds true.
	\end{remark}
	\begin{proof}
		Use \eqref{Tsplit} twice to get,
		\begin{align}
			\Theta(2,2,t,x) = \Theta(0,2,t+2,x) + 2x \Theta(1,1,t+2,x) + x^2 \Theta(2,0,t+2,x).  \label{remark1}
		\end{align}
		Use \eqref{Tinv} for the third term on the right-hand side of \eqref{remark1} and tend $t\to0$ to get,
		\begin{align*}
			\zeta^2(2) = \Theta(0,2,2,x) + 2x \Theta(1,1,2,x) + \Theta(0,2,2,\tfrac{1}{x}).
		\end{align*}
		Then, use Theorem \ref{THM-EVAL} thrice and simplify to get \eqref{newIDex}.
	\end{proof}
	
	 Such identitites are new in literature. We now give a family of mixed functional equations in the following theorem as a direct corollary of Theorem \ref{THM-EVAL}.
	\begin{theorem}\label{MIXEDthm}
		Let us define $\mathscr{F}(x)$ as follows,
		\begin{align*}
			\mathscr{F}(x):=& (-1)^r  \sum_{j=0}^{t-2} \binom{j+r-1}{j}  \frac{(-1)^{t-j}}{(t-j-1)!} \frac{1}{x^{j+r}} \sum_{m=1}^{\infty} \frac{\psi^{(t-j-1)}(mx+1)}{m^{2r+j}}  \notag \\ 
			&+ \sum_{i=0}^{r-2} \frac{(-1)^i}{x^{t+i}}  \binom{i+t-1}{i}  \zeta(r+t+i) \zeta(r-i) - \frac{(-1)^r}{x^{r+t-1}} \binom{r+t-2}{t-1}  \sum_{m=1}^{\infty} \frac{\gamma + \psi(mx+1)}{m^{2r+t-1}}.
		\end{align*}
		Then, we have,
		\begin{align}\label{mixed}
			\mathscr{F}(x) = x^{-t} \mathscr{F}\left(\frac{1}{x}\right)\hspace{-0.1 cm}.
		\end{align}
	\end{theorem}
	\begin{proof}
		Put $s=r$ in \eqref{T-H-EV} and substitute it in $\Theta(r,r,t,x)=x^{-t} \Theta(r,r,t,\tfrac{1}{x})$.
	\end{proof}
	\subsection{Kronecker limit type formula in the second variable}
	As an application of Theorem \ref{THM-EVAL}, we obtain the Kroncker limit formula for $\Theta(r,s,t,x)$ in the second argument $s$ in this subsection. This method is completely different from the one used in Section \ref{Sec2}.
	\begin{theorem}\label{KLF in s} The following hold:
	\begin{enumerate} \rm 
	\item \textit{For any  $r \in \mathbb{N},\ r \geq 2$,
		and for any $t \in \mathbb{N} \cup \{0 \}$, $s=1-t$ is a singular point of $\Theta(r,s,t,x)$, with the following polar singularity structure in $s$:}
		\begin{align} 
			&\Theta(r,s,t,x) = \frac{x^{-t}\zeta(r)}{(s-(1-t
				))} + \Bigg(  \gamma \zeta(r) x^{-t} + (-1)^r  \sum_{j=0}^{t-2} \binom{j+r-1}{j}  \frac{(-1)^{t-j}}{(t-j-1)!} \frac{1}{x^{j+r}} \sum_{m=1}^{\infty} \frac{\psi^{(t-j-1)}(mx+1)}{m^{r+j-t+1}}   \notag \\ & +  \sum_{i=1}^{r-2} \frac{(-1)^i}{x^{t+i}}  \binom{i+t-1}{i}  \zeta(i+1) \zeta(r-i) - \frac{(-1)^r}{x^{r+t-1}} \binom{r+t-2}{t-1}  \sum_{m=1}^{\infty} \frac{\gamma + \psi(mx+1)}{m^{r}} \Bigg) + O(|s-(1-t)|). \label{klfs1}
		\end{align}
	\item \textit{For $r=1$ and for any $t \in \mathbb{N}$, $s=1-t$ is a singular point of $\Theta(1,s,t,x)$, with the following polar singularity structure in $s$:}
		\begin{align}
			\Theta(1,s,t,x) &= \frac{x^{-t}}{(s-(1-t))^2} + \frac{x^{-t}}{(s-(1-t))} (\gamma+\log(x)+H_{t-1}) +\bigg( \frac{1}{x^{t}}(\gamma (\gamma+\log(x)+ H_{t-1}) + \gamma_1) \notag  \\
			& \left. -\sum_{j=0}^{t-2} \frac{(-1)^{t-j}x^{-j-1}}{(t-j-1)!}  \sum_{m=1}^{\infty} \frac{\psi^{(t-j-1)}(mx) + (-1)^{t-j}(t-j-2)!(mx)^{-(t-j-1)}}{m^{j-t+2}} \right. \notag \\
			& + \frac{t\zeta(2)}{x^{t+1}} +\frac{1}{x^t} \sum_{m=1}^{\infty} \frac{\psi(mx)-\log(mx)}{m} \bigg) + O(|s-(1-t)|).  \label{klfs2}
		\end{align}
	\end{enumerate}
	\end{theorem}
	\begin{proof}
	Let	$r \in \mathbb{N}, r \geq 2$,
		and $t \in \mathbb{N} \cup \{0 \}$. From \eqref{T-H-EV}, we have,
		\begin{align*}
			\Theta(r,s,t,x)- &\frac{1}{x^t} \zeta(s+t)\zeta(r) =  (-1)^r  \sum_{j=0}^{t-2} \binom{j+r-1}{j}  \frac{(-1)^{t-j}}{(t-j-1)!} \frac{1}{x^{j+r}} \sum_{m=1}^{\infty} \frac{\psi^{(t-j-1)}(mx+1)}{m^{r+s+j}}   \notag \\ & +  \sum_{i=1}^{r-2} \frac{(-1)^i}{x^{t+i}}  \binom{i+t-1}{i}  \zeta(s+t+i) \zeta(r-i) - \frac{(-1)^r}{x^{r+t-1}} \binom{r+t-2}{t-1}  \sum_{m=1}^{\infty} \frac{\gamma + \psi(mx+1)}{m^{r+s+t-1}}  . 
		\end{align*}
		Since we know, around $s=1-t$,
		\begin{align*}
			\zeta(s+t) = \frac{1}{s-(1-t)} +\gamma +O(|s-(1-t)|).
		\end{align*}
		Hence 
		\begin{align*}
			\lim_{s \to 1-t}	&\left( \Theta(r,s,t,x)- \frac{1}{x^t} \zeta(s+t)\zeta(r) \right) =  (-1)^r  \sum_{j=0}^{t-2} \binom{j+r-1}{j}  \frac{(-1)^{t-j}}{(t-j-1)!} \frac{1}{x^{j+r}} \sum_{m=1}^{\infty} \frac{\psi^{(t-j-1)}(mx+1)}{m^{r+j-t+1}}   \notag \\ & \hspace{2.5 cm} +  \sum_{i=1}^{r-2} \frac{(-1)^i}{x^{t+i}}  \binom{i+t-1}{i}  \zeta(i+1) \zeta(r-i) - \frac{(-1)^r}{x^{r+t-1}} \binom{r+t-2}{t-1}  \sum_{m=1}^{\infty} \frac{\gamma + \psi(mx+1)}{m^{r}}.
		\end{align*}
		Simplify to get, around $s=1-t$,
		\begin{align*} 
			&\Theta(r,s,t,x) = \frac{x^{-t}\zeta(r)}{(s-1+t)} + \Bigg(  \gamma \zeta(r) x^{-t} + (-1)^r  \sum_{j=0}^{t-2} \binom{j+r-1}{j}  \frac{(-1)^{t-j}}{(t-j-1)!} \frac{1}{x^{j+r}} \sum_{m=1}^{\infty} \frac{\psi^{(t-j-1)}(mx+1)}{m^{r+j-t+1}}   \notag \\ & +  \sum_{i=1}^{r-2} \frac{(-1)^i}{x^{t+i}}  \binom{i+t-1}{i}  \zeta(i+1) \zeta(r-i) - \frac{(-1)^r}{x^{r+t-1}} \binom{r+t-2}{t-1}  \sum_{m=1}^{\infty} \frac{\gamma + \psi(mx+1)}{m^{r}} \Bigg) + O(|s-1+t|).
		\end{align*}
		For $t-j-1\geq1$, $\psi^{(t-j-1)}(y)=O(y^{-(t-j-1)})$ as $y \to \infty$. Hence, we can see that, as $x \to \infty$, $$\frac{\psi^{(t-j-1)}(mx+1)}{m^{r+j-t+1}} = O\left(\frac{1}{m^r}\right).$$
		Since $r \geq 2$, the series involving the derivatives of $\psi(x)$ is absolutely convergent. This is the required result \eqref{klfs1}.
		
		 Let $r=1$ and for any $t \in \mathbb{N}$. From Theorem \ref{THM-EVAL} to get,
		\begin{align}
			\Theta(1,s,t,x) &= -  \sum_{j=0}^{t-2} \frac{(-1)^{t-j}}{(t-j-1)!} \frac{1}{x^{j+1}} \sum_{m=1}^{\infty} \frac{\psi^{(t-j-1)}(mx+1)}{m^{s+j+1}}  + \frac{1}{x^{t}}  \sum_{m=1}^{\infty} \frac{\gamma + \psi(mx+1)}{m^{s+t}}. \label{r=1 eval}
		\end{align}
		We simplify the terms on the right-hand side of \eqref{r=1 eval} as follows:
		\begin{align}
			\sum_{m=1}^{\infty} \frac{\gamma + \psi(mx+1)}{m^{s+t}} &= \gamma \zeta(s+t) + \sum_{m=1}^{\infty} \frac{\psi(mx+1)}{m^{s+t}} \notag \\&= \gamma \zeta(s+t) + \frac{1}{x}\zeta(s+t+1) + \sum_{m=1}^{\infty} \frac{\psi(mx)}{m^{s+t}}  \notag \\
			&= \gamma \zeta(s+t) + \frac{1}{x}\zeta(s+t+1) + \sum_{m=1}^{\infty} \frac{\psi(mx)-\log(mx)}{m^{s+t}} + \sum_{m=1}^{\infty} \frac{\log(mx)}{m^{s+t}} \notag \\
			&= (\gamma+\log(x)) \zeta(s+t) + \frac{1}{x}\zeta(s+t+1) -\zeta'(s+t) + \sum_{m=1}^{\infty} \frac{\psi(mx)-\log(mx)}{m^{s+t}}. \label{simp1}
		\end{align} 
		Simplification of the other term is as follows: Since $t-j-1\geq1$, we have $\psi^{(t-j-1)}(y+1)=\psi^{(t-j-1)}(y)-(-1)^{t-j}(t-j-1)!y^{j-t}$. Using it we can see,
		\begin{align} 
			\sum_{m=1}^{\infty} \frac{\psi^{(t-j-1)}(mx+1)}{m^{s+j+1}}	&=- \frac{(-1)^{t-j}(t-j-1)!}{x^{t-j}} \zeta(s+t+1)+ \sum_{m=1}^{\infty} \frac{\psi^{(t-j-1)}(mx)}{m^{s+j+1}} \notag \\
			&= - \frac{(-1)^{t-j}(t-j-1)!}{x^{t-j}} \zeta(s+t+1)- \frac{(-1)^{t-j}(t-j-2)!}{x^{t-j-1}} \zeta(s+t) \notag \\
			&\quad + \sum_{m=1}^{\infty} \frac{\psi^{(t-j-1)}(mx) + (-1)^{t-j}(t-j-2)!(mx)^{-(t-j-1)}}{m^{s+j+1}}. \label{simp2}
		\end{align}
		Substitute \eqref{simp1} and \eqref{simp2} in \eqref{r=1 eval} to get,
		\begin{align*}
			&\Theta(1,s,t,x) - \frac{1}{x^t} \Bigg(\gamma+\log(x)+  \sum_{j=0}^{t-2} \frac{1}{(t-j-1)} \Bigg) \zeta(s+t) + \frac{1}{x^t} \zeta'(s+t)  \notag \\
			&= \frac{t}{x^{t+1}}\zeta(s+t+1) +\frac{1}{x^t} \sum_{m=1}^{\infty} \frac{\psi(mx)-\log(mx)}{m^{s+t}} \notag \\
			&\quad -  \sum_{j=0}^{t-2} \frac{(-1)^{t-j}}{(t-j-1)!} \frac{1}{x^{j+1}} \sum_{m=1}^{\infty} \frac{\psi^{(t-j-1)}(mx) + (-1)^{t-j}(t-j-2)!(mx)^{-(t-j-1)}}{m^{s+j+1}}.
		\end{align*}
		Taking limit as $s \to 1-t$, we get,
		\begin{align*}
			&\lim_{s \to 1-t} \left( \Theta(1,s,t,x) - \frac{1}{x^t} \left(\gamma+\log(x)+ H_{t-1} \right) \zeta(s+t) + \frac{1}{x^t} \zeta'(s+t) \right) \\
			&= \frac{t}{x^{t+1}}\zeta(2) +\frac{1}{x^t} \sum_{m=1}^{\infty} \frac{\psi(mx)-\log(mx)}{m} \notag \\
			&\quad - \sum_{j=0}^{t-2} \frac{(-1)^{t-j}}{(t-j-1)!} \frac{1}{x^{j+1}} \sum_{m=1}^{\infty} \frac{\psi^{(t-j-1)}(mx) + (-1)^{t-j}(t-j-2)!(mx)^{-(t-j-1)}}{m^{j-t+2}}.
		\end{align*}
		%		\begin{align*}
			%			&\lim_{s \to 1-t} \left( \Theta(1,s,t,x) - \frac{1}{x^t} \left(\gamma+\log(x)+ H_{t-1} \right) \zeta(s+t) + \frac{1}{x^t} \zeta'(s+t) \right) \\
			%			&= \frac{t}{x^{t+1}}\zeta(2) +\frac{1}{x^t} \sum_{m=1}^{\infty} \frac{\psi(mx)-\log(mx)}{m} \notag \\
			%			&-\sum_{j=0}^{t-2} \frac{(-1)^{t-j}}{(t-j-1)!} \frac{1}{x^{j+1}} \sum_{m=1}^{\infty} \frac{\psi^{(t-j-1)}(mx) + (-1)^{t-j}(t-j-2)!(mx)^{-(t-j-1)}}{m^{j-t+2}},
			%		\end{align*}
		where $H_n$ is the $n^{th}$ Harmonic number, with the convention $H_0=0$. Also, see that, as $y \to \infty$,
		\begin{align*}
			\psi^{(t-j-1)}(y) + (-1)^{t-j}(t-j-2)! \frac{1}{y^{t-j-1}} = O(y^{-(t-j)}),
		\end{align*}
		which ensures the convergence of the series involving the derivatives of $\psi(x)$. Use the Laurent expansions of $\zeta(s+t)$ and $\zeta'(s+t)$ around $s=1-t$ to get the required result \eqref{klfs2}.
	\end{proof}
	\begin{corollary}\label{s cor}
		Around $s=0$, we have,
		\begin{align*}
			\Theta(1,s,1,x) = \frac{1}{xs^2} +\frac{\gamma+\log(x)}{xs} +\frac{1}{x}(\gamma(\gamma+\log(x))+\gamma_1) +\frac{\zeta(2)}{x^{2}}  +\frac{1}{x} \sum_{m=1}^{\infty} \frac{\psi(mx)-\log(mx)}{m}+O(s).
		\end{align*}
	\end{corollary}
	\begin{proof}
		Put $r=t=1$ in Theorem \ref{KLF in s} and tend $s \to 0$ to get,
		\begin{align*}
			\lim_{s \to 0} \left( \Theta(1,s,1,x) - \frac{1}{x} \left(\gamma+\log(x) \right) \zeta(s+1) + \frac{1}{x} \zeta'(s+1) \right) = \frac{\zeta(2)}{x^{2}}  +\frac{1}{x} \sum_{m=1}^{\infty} \frac{\psi(mx)-\log(mx)}{m}.
		\end{align*}
		Use the Laurent series expansions of $\zeta(s+1)$ and $\zeta'(s+1)$ around $s=0$ to get the result. 
	\end{proof}
	%		\textcolor{red}{Verift the following result if it is valid!}
	%		\begin{corollary}
		%			Around $s=1$, we have
		%			\begin{align*}
			%				\Theta(1,s,0,x) = -\frac{1}{(s-1)^2} + \frac{\gamma+\log(x)}{s-1} + \gamma^2 +\gamma\log(x) + \gamma_1 + \sum_{m=1}^{\infty} \frac{\psi(mx)-\log(mx)}{m}+O(|s-1|).
			%			\end{align*}
		%		\end{corollary}
	%		\begin{proof}
		%			Put $r=1$ and $t=0$ in Theorem \ref{KLF in s} to get,
		%			\begin{align*}
			%				\lim_{s \to 1} \bigg( \Theta(1,s,0,x) -  \left(\gamma+\log(x) \right) \zeta(s) +  \zeta'(s) \bigg)= \sum_{m=1}^{\infty} \frac{\psi(mx)-\log(mx)}{m}.
			%			\end{align*}
		%			Use the Laurent series expansions of $\zeta(s)$ and $\zeta'(s)$ around $s=1$ to get the required result.
		%		\end{proof}
	\begin{remark}
		From \cite[Equation 7.12]{zagier},
		\begin{align*}
			\sum_{m=1}^{\infty} \frac{ \psi(m) -\log(m)}{m} = -\frac{\gamma^2}{2} -\frac{\pi^2}{12}  -\gamma_1,
		\end{align*}
	    and putting $x=1$ in Corollary \ref{s cor},  we can see that around $s=0$,
		\begin{align}
			\Theta(1,s,1,1) = \frac{1}{s^2} + \frac{\gamma}{s} + \frac{6\gamma^2+\pi^2}{12}    +O(|s|). \label{REMARK2}
		\end{align}
		%			\textup{\textcolor{red}{Add two equations after Theorem 1.1 from the previous paper to compare and contrast the behaviour of $\Theta$ as it approaches $(1,1,0,1)$ from different directions.}}
	\end{remark}
	\section{Special values of the function $\Theta(s,s,s,x)$}\label{RomikSection}
 	In this section, we will study the special values at points of indeterminacy and the singularities of $\Theta(s,s,s,x)$, in the variable $s$. Recall that $	\Theta(s_0, s_0, s_0, x)_{L_1} := \lim_{s \to s_0} \Theta(s,s,s,x)$. We have the following theorem which is a generalization of the evaluations in \eqref{sv omega1} and \eqref{sv omega2}.
	
	\begin{theorem}\label{SV} \rm 
		\textit{Let $j$ be a non-negative integer. We have the following special values for $\Theta(s,s,s,x)$: }
		\begin{enumerate}
%			\item \textit{For $s=0$,}
		%	\begin{align}
		%		\Theta(0,0,0,x) &=\frac{1+6x+x^2}{24x}. \label{svT1}
		%	\end{align} 
			\item \textit{For $j$ odd,}
			\begin{align*}
			\Theta(-j,-j,-j,x)_{L_1} &=0. 
			\end{align*}
			\item \textit{For $j$ even,}
			\begin{align}
			\Theta(-j,-j,-j,x)_{L_1} &= - \frac{   (j!)^2}{ 2 (2j+1)!}(x^{2j+1}+x^{-j-1})\zeta(-3j-1)     + \sum_{\substack{k=0 }}^{j}  \binom{j}{k} x^{j-k} \zeta(-j-k)\zeta(-2j+k). \label{svT2}
			\end{align} 
		\end{enumerate}
%		\begin{align*}
%			\Theta(0,0,0,x) &=\frac{1+x+6x^2}{24},\\
%			\Theta(-j,-j,-j,x) &=0, \hspace{3 cm} (\textup{for odd positive integer } j),\\
%			\Theta(-j,-j,-j,x) &= (x^{2j+1}+x^{-j-1}) \frac{ (-1)^{j+1}  (j!)^2}{ 2 (2j+1)!}\zeta(-3j-1)     + \sum_{\substack{k=0 }}^{j}  \binom{j}{k} \frac{\zeta(-j-k)\zeta(-2j+k)}{ x^{j+k}}
%		\end{align*} 
%			$\Theta(0,0,0,x) =\frac{1+x+6x^2}{24}$.
%				$\Theta(-j,-j,-j,x) =0$, for odd positive integer $j$.
%		We have,
%		\begin{itemize}
%			\item 	$\Theta(0,0,0,x) =\frac{1+x+6x^2}{24}$.\\
%			\item 	$\Theta(-j,-j,-j,x) =0$, for odd positive integer $j$.
%		\end{itemize}

	\end{theorem}
\begin{proof}
Taking $r=t=s$ in Proposition \ref{first}, for $s> \frac{2}{3}$,  $M \in \mathbb{N} \cup \{0\} $ and $x>0$, we have  
\begin{align}\label{RomikInt}
	&\Theta(s,s,s,x) -\frac{1}{\Gamma(s)} \int_{0}^{\infty} y^{s-1} \textup{Li}_s (e^{-xy}) \Bigg( \textup{ Li}_s (e^{-y}) - \Gamma(1-s) y^{s-1} - \sum_{\substack{k=0 }}^{M} (-1)^k\zeta(s-k) \frac{y^k}{k!}\Bigg) \, dy \notag \\
	&=   \frac{\Gamma(1-s)\Gamma(2s-1)}{ x^{2s-1}\Gamma(s)}\zeta(3s-1)    + \sum_{\substack{k=0 }}^{M} \frac{ (-1)^k}{k!} \frac{\zeta(s-k)\Gamma(s+k)}{ x^{s+k}\Gamma(s)} \zeta(2s+k).
\end{align} 
For $\Re(s) \ge \frac{-M}{2}+\epsilon, \epsilon>0$,
\begin{align}\label{romikbound}
	y^{s-1} \textup{Li}_s (e^{-xy}) \Bigg( \textup{ Li}_s (e^{-y}) - \Gamma(1-s) y^{s-1} - \sum_{\substack{k=0 }}^{M} (-1)^k\zeta(s-k) \frac{y^k}{k!}\Bigg)= O(y^{2 \epsilon-1}).
\end{align}
Therefore the integral is convergent for $\Re(s) \ge \frac{-M}{2}+\epsilon$ and hence analytic.
%Choose $M=1$.  As $s\to 0$, the term involving the integral in \eqref{RomikInt} vanishes due to $\Gamma(s)$. Expanding the zeta and gamma factors of right-hand side terms around $s=0$ gives us \eqref{svT1}. 
For any fixed non-negative integer $j$, choose $M=2j+1$, and see that the term involving the integral in \eqref{RomikInt} vanishes as $s\to -j$. Moreover, the only terms which survive in the finite sum are when $0\le k \le j $ and $k=2j+1$. Therefore, using the duplication formula for $\Gamma(s)$ in the first term of right-hand side of \eqref{RomikInt}, we have 
\begin{align*}
	\Theta(-j,-j,-j,x)_{L_1} = &  \frac{ x^{2j+1} \Gamma(-j-\frac{1}{2}) \Gamma(j+1)}{\sqrt{\pi} 2^{2j+2}}\zeta(-3j-1)  \\&  + \sum_{\substack{k=0 }}^{j} \frac{ (-1)^k}{k!} (-j) (-j+1)\cdots(-j+k-1) x^{j-k} \zeta(-j-k) \zeta(-2j+k)\\
	& +\frac{ (-1)^{2j+1}}{(2j+1)!} \frac{\zeta(-3j-1)\Gamma(j+1)}{ x^{j+1}} \frac{(-1)^j j!}{2}\\
	= &  \frac{ (-1)^{j+1}x^{2j+1}  (j!)^2}{ 2 (2j+1)!}\zeta(-3j-1)  + \frac{ (-1)^{j+1}(j!)^2}{2 (2j+1)! x^{j+1}} \zeta(-3j-1) \\&  + \sum_{\substack{k=0 }}^{j}  \binom{j}{k} x^{j-k} \zeta(-j-k)\zeta(-2j+k) .
\end{align*}
For odd values of $j$, the right--hand side of above expression equals zero, due to trivial zeros of $\zeta(s)$ at negative even integers. For $j$ even, we get \eqref{svT2} upon simplification. 
\end{proof}
\begin{remark}\label{RMK-jzero}
	\rm
	For $j>0$ even, we can use the following recursive identity \cite[p.132]{romik}
	\begin{align*}
		\frac{j!}{(2j+1)!}\zeta(-3j-1) = \sum_{i=0}^{j} \frac{1}{i! (j-i)!} \zeta(-i-j)\zeta(i-2j),
	\end{align*}
	to see that $\Theta(-j,-j,-j,1)_{L_1}=0$, i.e., \eqref{sv omega2} holds for any positive integer $j$.
\end{remark}
\begin{theorem}\label{TTheorem}
	\rm
	\textit{The function $\Theta(s,s,s,x)$ has singularities exactly at $s=\frac{2}{3}$ and $s=\frac{1}{2}-j$, for any $j \in \mathbb{N} \cup\{0\} $. Moreover,} 
	\begin{enumerate}
		\item \textit{around $s=\frac{2}{3}$, we have}
		\begin{align}
			\Theta(s,s,s,x)=\frac{\Gamma^3 \hspace{-0.07 cm}\left(\frac{1}{3}\right)}{2\sqrt{3}\pi x^\frac{1}{3} \left(s-\frac{2}{3}\right)}+O(1). \label{T1}
		\end{align} 
		\item \textit{around $s=\frac{1}{2}-j$, we have}
%		\begin{align}
%			\Theta(s,s,s,x)= \frac{(x^{2j}+x^{-\frac{1}{2}- j}) \Gamma\left(\frac{1}{2}+j\right) \zeta\left(\frac{1-6j}{2} \right)}{2 (2 j)! \Gamma\left(\frac{1}{2}-j\right)
%				\left(s - \frac{1}{2} + j\right)} +O(1). \label{T21}
%		\end{align}
		\begin{align}
			\Theta(s,s,s,x)= (-1)^j 2^{-4j-1}  \binom{2j}{j} \left(x^{2j}+x^{-\frac{1}{2}- j}\right)   \frac{\zeta\left(\frac{1-6j}{2} \right)}{
				\left(s - \left( \frac{1}{2} - j \right)\right)} +O(1). \label{T2}
		\end{align}
	\end{enumerate}
\end{theorem}
\begin{proof}
	From \eqref{pole conditions}, we know that the function $\Theta(s,s,s,x)$ has its singularities only possible at $s=\frac{2}{3}$ or $s=\frac{1-\ell}{2}$, for any non-negative integer $\ell$. At $s=\frac{2}{3}$, the finite sum on the right-hand side of \eqref{RomikInt} has no singularity. Also, note that taking $M=0$ in \eqref{RomikInt} suffices. Therefore the only pole arising in the right-hand side of \eqref{RomikInt} is due to $\zeta(3s-1)$. After using the Laurent series expansion of the other elementary functions, around $s =\frac{2}{3}$, we have,
\begin{align*}
	\Theta(s,s,s,x)=\frac{\Gamma^2\hspace{-0.07 cm}\left(\frac{1}{3}\right)}{3 x^\frac{1}{3}\Gamma\hspace{-0.07 cm}\left(\frac{2}{3}\right) \left(s-\frac{2}{3}\right)}+O(1).
\end{align*} 
% Now choose $M=\ell$. Then for $\Re(s) \ge \frac{-\ell}{2}+\epsilon, \epsilon>0$,
% \begin{align*}
% y^{s-1} \textup{Li}_s (e^{-xy}) \Bigg( \textup{ Li}_s (e^{-y}) - \Gamma(1-s) y^{s-1} - \sum_{\substack{k=0 }}^{M} (-1)^k\zeta(s-k) \frac{y^k}{k!}\Bigg)= O(y^{2 \epsilon-1}).
 %\end{align*}
 %Therefore the integral is convergent for $\Re(s) \ge \frac{-\ell}{2}+\epsilon$ and hence analytic.
 Using the reflection formula for $\Gamma(s)$, one can see that $\Gamma(\frac{1}{3})\Gamma(\frac{2}{3})= \frac{2\pi}{\sqrt{3}}$, and hence get \eqref{T1}.
 
 If $\ell=2j-1$, for some $j \in \mathbb{N}$, then $\frac{1-\ell}{2}$ will be a non-positive integer. Hence, from Theorem \ref{SV}, clearly, $\Theta(s,s,s,x)$ has no singularity at $s=\frac{1-\ell}{2}$. Therefore we assume $\ell= 2 j$, for some $j \in \mathbb{N} \cup \{0\}$. Taking $M=2j$, the integral in \eqref{RomikInt} converges in $\Re(s) \ge j+\epsilon, \epsilon>0$ as shown in \eqref{romikbound}. At $s=\frac{1}{2}-j$, note that $\Gamma(2s-1)$ is the only polar factor in the first term of the right hand side of \eqref{RomikInt}. Hence, around $s=\frac{1}{2}-j$
 \begin{align}\label{rom1}
  \frac{\Gamma(1-s)\Gamma(2s-1)}{ x^{2s-1}\Gamma(s)}\zeta(3s-1)  = \frac{x^{2 j} \Gamma\left(\frac{1}{2}+j\right) \zeta\left(\frac{1-6j}{2} \right)}{2 (2 j)! \Gamma\left(\frac{1}{2}-j\right)
  	\left(s - \frac{1}{2} + j\right)} +O(1).\end{align}
  In the finite sum \eqref{RomikInt}, the only polar term is when $k=2j$, in which case the zeta factor $\zeta(s+2j)$ has pole at $s=\frac{1}{2}-j$. Therefore, writing the Laurent series expansion, we have
  \begin{align}\label{rom2}
  	\sum_{\substack{k=0 }}^{M} \frac{ (-1)^k}{k!} \frac{\zeta(s-k)\Gamma(s+k)}{ x^{s+k}\Gamma(s)} \zeta(2s+k)=  \frac{x^{-\frac{1}{2}- j} \Gamma\left(\frac{1}{2}+j\right) \zeta\left(\frac{1-6j}{2} \right)}{2 (2j)! \Gamma\left(\frac{1}{2}-j\right)
  		\left(s - \frac{1}{2} + j\right)} +	O(1).
  \end{align}
  Finally, from \eqref{RomikInt},\eqref{rom1} and \eqref{rom2}, we get
  \begin{align*}
  	\Theta(s,s,s,x)= \frac{(x^{2j}+x^{-\frac{1}{2}- j}) \Gamma\left(\frac{1}{2}+j\right) \zeta\left(\frac{1-6j}{2} \right)}{2 (2j)! \Gamma\left(\frac{1}{2}-j\right)
  		\left(s - \frac{1}{2} + j\right)} +O(1),
  \end{align*}
  around $s=\frac{1}{2}-j$. Using the reflection formula and the duplication formula for $\Gamma(s)$, one can see that
  \begin{align*}
  	\frac{ \Gamma\left(\frac{1}{2}+j\right)}{(2j)! \Gamma\left(\frac{1}{2}-j\right)} = (-1)^j 2^{-4j} \binom{2j}{j},
  \end{align*}
  which gives us \eqref{T2}.
  \end{proof}
\section{Concluding remarks}
\noindent We end the paper with following remarks and problems:
\begin{enumerate}
	 \item In this paper, we have obtained the Kronecker limit formula for the generalized Mordell-Tornheim zeta function $\Theta(r,s,t,x)$ in the second and third variables with other parameters fixed. However, it would be interesting to study the behavior of $\Theta(r,s,t,x)$ around points of indeterminacy, as a function of three complex variables with various limiting approaches. One should find special values of $\Theta(r,s,t,x)$ at points $r=s=t=-j$, for any $j \in \mathbb{N}$, with a limiting sense different from the one we take in \eqref{sense}. For instance, one can refer to \cite[Example 1]{Komori} for a detailed example.
	 \item Theorem \ref{TTheorem} gives only the principal part of the Laurent series expansion of $\Theta(s,s,s,x)$ around its singularities. However, it is desirable to further obtain the constant term of the Laurent expansion explicitly. 
	 \item  For a fixed $x$, $\Theta(s,s,s,x)$ is a function of one variable complex $s$, with an absolutely convergent double series representation for $\Re(s)>\frac{2}{3}$, which can meromorphically extended to the whole complex plane. As asked in \cite[Section 5]{zhao}, it will be interesting to determine if $\Theta(s,s,s,x)$ satisfies a functional equation analogous to that of the $\zeta(s)$.
	 
%	 Finding a Hurwitz zeta type functional equation which connects $\Theta(s,s,s, x)$ and $\Theta(1-s,1-s,1-s, x)$ will be of great value. 
	 In Theorem \ref{SV}, we have shown that $\Theta(s,s,s,x)$ has a zero at negative odd integers  $s$. However, finding other non-trivial zeros of this function is an interesting open problem. A variant of this question was also asked in \cite[Section 5]{zhao}.
	\item The study of the Mordell-Tornheim zeta function $\zeta_{\textup{MT}}(r,s,t)$ with an additional parameter $x$ gives us the two-term functional equations for Herglotz-Zagier function as shown in \cite[Theorem 1.3]{dss2}. It also gives new functional equations \eqref{mixed}. Hence, this warrants an independent study on the general Witten-zeta function $\zeta_\Phi(s)$ \cite{KCA}, associated to a root system $\Phi$, with a parameter $x$, leading to new interesting connections.

\end{enumerate}
\section{Acknowledgment}
	The authors would like to extend gratitude to Atul Dixit for his insights and continuous support. The first author was partially supported by the grant MIS/IITGN/R\&D/MATH/AD/2324/058 of IIT Gandhinagar, and the Seed money grant CUK/ACAD-II/F-3787/26 of Central University of Karnataka. The second author thanks IIT Gandhinagar for partially funding the research. \\

\noindent \textbf{Data Availability Statement:} The manuscript has no associated data with it.

\end{document}